\documentclass[reqno]{amsart}
\usepackage{graphicx}
\usepackage[dvips]{epsfig}
\usepackage{bm}
\usepackage{mathrsfs}
\usepackage{epstopdf}

 \makeatletter
\@namedef{subjclassname@2020}{%
  \textup{2020} Mathematics Subject Classification}
\makeatother

\newtheorem{theorem}{Theorem}[section]
\newtheorem{lemma}{Lemma}[section]
\newtheorem{proposition}[lemma]{Proposition}
\newtheorem{definition}[lemma]{Definition}
\newtheorem{remark}[lemma]{Remark}
\newtheorem{problem}{{\bf{Problem}}}[section]
\numberwithin{figure}{section}

\numberwithin{equation}{section}

\allowdisplaybreaks
\begin{document}

\title[Supersonic flows for the Euler-Poisson system]{Three-dimensional Supersonic  flows for the steady Euler-Poisson system in divergent nozzles}

\author{Hyangdong Park}
\address{School of Mathematics, Korea Institute for Advanced Study (KIAS), 85 Hoegi-ro, Dongdaemun-gu, Seoul, 02455, Republic of Korea}
\email{hyangdong@kias.re.kr}




\date{\today}

\begin{abstract}
We are concerned with the unique existence of an axisymmetric supersonic solution with nonzero vorticity and nonzero angular momentum density for the steady Euler-Poisson system in three-dimensional divergent nozzles when prescribing the velocity, strength of electric field, and the entropy at the entrance. We first reformulate the problem via the method of the Helmholtz decomposition for three-dimensional axisymmetric flows and obtain a solution to the reformulated problem by the iteration method. Furthermore, we deal carefully with singularity issues related to the polar angle on the axis of the divergent nozzle.
\end{abstract}

\maketitle


\section{Introduction}
The steady Euler-Poisson system
\begin{equation}\label{E-S}
\left\{
\begin{split}
&\mbox{div}(\rho{\bf u})=0,\\
&\mbox{div}(\rho{\bf u}\otimes{\bf u})+\nabla p=\rho\nabla\Phi,\\
&\mbox{div}(\rho\mathcal{E}{\bf u}+p{\bf u})=\rho{\bf u}\cdot\nabla\Phi,\\
&\Delta\Phi=\rho-b,
\end{split}
\right.
\end{equation}
describes a hydrodynamical model of semiconductor devices or plasmas.
In \eqref{E-S}, the functions $\rho=\rho({\bf x})$, ${\bf u}=(u_1,u_2,u_3)({\bf x})$, $p=p({\bf x})$, and $\mathcal{E}=\mathcal{E}({\bf x})$ represent the macroscopic particle electron density, velocity, pressure, and the total energy, respectively, at ${\bf x}=(x_1,x_2,x_3)\in\mathbb{R}^3$.
The function $\Phi=\Phi({\bf x})$ is the electric potential generated by the Coulomb force of particles.
The function $b=b({\bf x})>0$ represents the density of fixed, positively charged background ions. 
In this paper, we consider the case for which the pressure $p$ and the total energy $\mathcal{E}$ are given by 
\begin{equation*}
p(\rho, S)=S\rho^{\gamma}\quad\mbox{and}\quad 
\mathcal{E}(\rho,|{\bf u}|,S)=\frac{|{\bf u}|^2}{2}+\frac{1}{\gamma-1}S\rho^{\gamma},
\end{equation*}
respectively, where the function $S=S({\bf x})$ represents the entropy and the constant $\gamma>1$ is the {\emph {adiabatic constant}}.
Then the system \eqref{E-S} can be rewritten as 
\begin{equation}\label{E-B}
\left\{
\begin{split}
&\mbox{div}(\rho{\bf u})=0,\\
&\mbox{div}(\rho{\bf u}\otimes{\bf u})+\nabla p=\rho\nabla\Phi,\\
&\mbox{div}(\rho{\bf u}B)=\rho{\bf u}\cdot\nabla\Phi,\\
&\Delta\Phi=\rho-b.
\end{split}
\right.
\end{equation}
In the above, the function $B$ is called the {\emph{Bernoulli's function}} which is defined by
\begin{equation}\label{Ber-inv}
B(\rho,|{\bf u}|,S):=\frac{|{\bf u}|^2}{2}+\frac{\gamma }{\gamma-1}S\rho^{\gamma-1}.
\end{equation}
From the system \eqref{E-B}, we can see that
\begin{equation*}
\rho{\bf u}\cdot\nabla S=\rho{\bf u}\cdot\nabla \mathcal{K}=0
\end{equation*}
for the {\emph{pseudo-Bernoulli's invariant}} $\mathcal{K}=\mathcal{K}({\bf x})$ defined by 
\begin{equation*}
\mathcal{K}:=B-\Phi.
\end{equation*}
Since we will consider nonzero entropy $S$, we assume that $\mathcal{K}\equiv 0$  for simplicity throughout the paper.

Let $(r,\phi,\theta)$ be the spherical coordinates of ${\bf x}=(x_1,x_2,x_3)\in\mathbb{R}^3$, that is, 
\begin{equation}\label{s-coord}
(x_1,x_2,x_3)=(r\sin\phi\cos\theta,r\sin\phi\sin\theta,r\cos\phi),\quad r\ge 0,\, \phi\in[0,\pi],\,\theta\in\mathbb{T},
\end{equation}
where $\mathbb{T}$ is a one dimensional torus with period $2\pi$.
%
Let $D$ be an open subset of $\mathbb{R}^3$ and let ${\bf x}=(x_1,x_2,x_3)\in D$.
\begin{definition}
\begin{itemize}
\item[(i)] A function $f:D\to\mathbb{R}$ is {\emph{axisymmetric}} if $f({\bf x})=\tilde{f}(r,\phi)$ for some $\tilde{f}:\mathbb{R}^2\to\mathbb{R}$ so its value is independent of $\theta$.
\item[(ii)] A vector-valued function ${\bf F}:D\to\mathbb{R}^3$ is {\emph{axisymmetric}} if ${\bf F}({\bf x})=F_r(r,\phi){\bf e}_r+F_\phi(r,\phi){\bf e}_\phi+F_{\theta}(r,\phi){\bf e}_\theta$ for some axisymmetric functions $F_r$, $F_\phi$, and $F_\theta$ and an orthonormal basis $\{{\bf e}_r,{\bf e}_{\phi},{\bf e}_\theta\}$ defined by 
\begin{equation*}
\begin{split}
&{\bf e}_r:=(\sin\phi\cos\theta,\sin\phi\sin\theta,\cos\phi),\\
 &{\bf e}_\phi:=(\cos\phi\cos\theta,\cos\phi\sin\theta,-\sin\phi),\\
&{\bf e}_{\theta}:=(-\sin\theta,\cos\theta, 0). 
\end{split}
\end{equation*}
\end{itemize}
\end{definition}
%
%
The goal of this work is to prove
the unique existence of an axisymmetric supersonic ($|{\bf u}|>\sqrt{\gamma S\rho^{\gamma-1}}$) solutions with nonzero vorticity ($\nabla\times{\bf u}\ne 0$)  and nonzero angular momentum density (${\bf u}\cdot{\bf e}_{\theta}\ne 0$) for the steady Euler-Poisson system \eqref{E-B} in a three-dimensional divergent nozzle  $\Omega$ defined by 
\begin{equation*}\label{domain}
\Omega:=\left\{{\bf x}=(x_1,x_2,x_3)\in\mathbb{R}^3:\, 1<r_{\rm en}<r<r_{\rm ex},\,\phi\in[0,\phi_0)\right\}\quad\mbox{(Figure \ref{fig-pro})}
\end{equation*}
for positive constants $r_{\rm en}$, $r_{\rm ex}\in\mathbb{R}$ and $\phi_0\in(0,\pi)$ 
when prescribing the velocity, strength of electric field, and the entropy at the entrance.
The length $L:=r_{\rm ex}-r_{\rm en}$ of the nozzle will be determined later.

\begin{figure}[!h]
\centering
\includegraphics[scale=0.65]{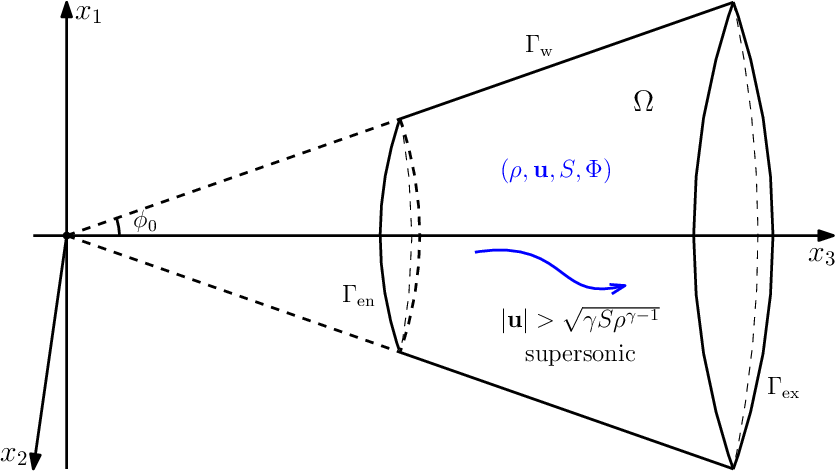}
\caption{Flows in a divergent nozzle $\Omega$}\label{fig-pro}
\end{figure}



Very recently, the stability of   supersonic  solutions to the steady Euler-Poisson system in two-dimensional divergent nozzles and three-dimensional cylinders were studied in \cite{Duan2023super,bae2025supersonic}, respectively.
In \cite{Duan2023super}, the authors studied the stability of potential flows  by  applying the method of \cite{bae2021structural} to the divergent nozzle problem with the standard polar coordinates $(r,\theta)$. 
In this paper, we study flows with nonzero vorticity in a three-dimensional space, so we  deal carefully with singularity issues related with the polar angle on the axis of the divergent nozzle.
%
%
In \cite{bae2025supersonic}, the unique existence of three-dimensional supersonic irrotational solutions in a cylindrical nozzle with an arbitrary cross section and axisymmetric solutions with nonzero vorticity in a circular cylinder were studied.
One of the differences  from \cite{bae2025supersonic} is that here we do not use the method of reflection across the entrance, so we do not assume the zero entrance data of the background electric force. 
Moreover, we do not use a weighted norm to obtain the classical solution. 
We first reformulate the problem via the  Helmholtz decomposition method for axisymmetric flows in divergent nozzles and obtain a unique strong solution with a standard Sobolev norm. 
Then, using the fact that the longitudinal axis is the axis of the time-like variable, we apply the standard elliptic theory for the interior regularity to see that the solution is a classical solution.
The new feature of this work is that it is not only the first to study the three-dimensional divergent nozzle problem, but also proves the stability of the three-dimensional supersonic flows to the steady Euler-Poisson system in a different way from  \cite{bae2025supersonic}. 


To the best of our knowledge, there are very few mathematical results on multi-dimensional transonic shocks for the steady Euler-Poisson system (cf. \cite{gamba1996viscous,luo2011stability,luo2012transonic}).
The ultimate goal of this work is to study the existence and stability of  multi-dimensional transonic shocks for the Euler-Poisson system in various domains. 
The geometry of domains is quite important in study about transonic shocks, because the behavior of the flow depends on it. 
Thus this study will be an important basis for studying transonic shock solutions to the steady Euler-Poisson system in divergent nozzles.
For more information on transonic shocks, one may refer to \cite{chen2018mathematics,courant1999supersonic} and references cited therein.

The rest of the paper is organized  as follows. 
In Section \ref{Sec-Main}, problem and main result are stated. 
In Section \ref{Sec-HD}, the main theorem is restated for the reformulated problem via the method of the Helmholtz decomposition. 
In Section \ref{last-sec}, assuming that a lemma is true, we prove the main theorem.  
In Section \ref{sec-lem-last}, we provide a proof of the lemma assumed in Section 4. 

\section{Problems and Main Theorems}\label{Sec-Main}

\subsection{One-dimensional radial solutions}
Suppose that 
\begin{equation*}
b=\bar{b}(r)
\end{equation*}
for a fixed function $\bar{b}>0$, and suppose that 
\begin{equation*}
{\bf u}=u(r){\bf e}_r,\quad \rho=\rho(r),\quad p=p(r),\quad \Phi=\Phi(r)
\end{equation*}
with $u>0$, $\rho>0$, and $p>0$.
Let us set 
\begin{equation*}
E:=\Phi',
\end{equation*}
where $'$ denotes the derivative with respect to $r$.
Then the Euler-Poisson system \eqref{E-B} can be rewritten as follows:
\begin{equation}\label{F-sys}
\left\{\begin{split}
&(r^2\rho u)'=0,\\
&(r^2\rho u^2)'+r^2p'=r^2\rho E,\\
&(r^2\rho  u B)'=r^2\rho u E,\\
&(r^2E)'=r^2(\rho-\bar{b}).
\end{split}\right.
\end{equation}
From \eqref{F-sys}, one can see that the following properties hold:
\begin{itemize}
\item $r^2\rho u=m_0\mbox{ for a constant $m_0>0$;}$
\item $S=S_0\mbox{ for a constant $S_0>0$;}$
\item  $B'=E$.
\end{itemize}
%
Then the system \eqref{F-sys} can be rewritten as an ODE system for $(\rho, E)$ as follows:
\begin{equation}\label{ODE-RE}
\left\{\begin{split}
&\rho'=g_1(r,\rho, E),\\
&E'=g_2(r,\rho,E)
\end{split}\right.
\end{equation}
for $g_1$ and $g_2$ defined by 
\begin{equation*}
\begin{split}
&g_1(r,\rho,E):=\left(r^2\rho E+\frac{2m_0^2}{r^3\rho}\right)\left(r^2\gamma S_0\rho^{\gamma-1}-\frac{m_0^2}{r^2\rho^2}\right)^{-1},\\
&g_2(r,\rho,E):=(\rho-\bar{b})-\frac{2E}{r}.
\end{split}
\end{equation*}
Equivalently, the system \eqref{F-sys} can be rewritten as an ODE system for $(M:=|u|/\sqrt{\gamma S\rho^{\gamma-1}}, E)$ as follows: 
\begin{equation}\label{ODE-ME}
\left\{\begin{split}
&(M^2)'=h_1(r,M,E),\\
&(r^2E)'=h_2(r,M)
\end{split}\right.
\end{equation}
for
\begin{equation*}
\begin{split}
&h_1(r,M,E):=\frac{M^2}{M^2-1}\left[\frac{2}{r}(2+(\gamma-1)M^2)+\frac{1}{m_0^2}(\gamma+1)\kappa_0^2(r^4M^2)^{\frac{\gamma-1}{\gamma+1}}E\right],\\
&h_2(r,M):=r^2\left(\kappa_0\left(\frac{1}{r^4M^2}\right)^{\frac{1}{\gamma+1}}-\bar{b}\right)\quad\mbox{with }\kappa_0:=\left(\frac{m_0^2}{\gamma S_0}\right)^{\frac{1}{\gamma+1}}.
\end{split}
\end{equation*}
One may refer to \cite[Section 2]{bae2018radial} for detailed derivations.

For the ODE system \eqref{ODE-RE} (or \eqref{ODE-ME}), we 
assume that $\bar{b}\equiv b_0$ for a fixed constant $b_0>0$, and we consider the following initial conditions:
\begin{equation}\label{ini-con-super}
(\rho,  E)(r_{\rm en})=(\rho_0,  E_0)\,\,\mbox{ or }\,\, (E,M)(r_{\rm en})=(E_0, M_0)
\end{equation}%
for constants $\rho_0$, $E_0$, and $M_0$ satisfying
\begin{equation}\label{ini-con-super-1}
\begin{split}
&0<\rho_0<\left(\frac{m_0^2}{\gamma r_{\rm en}^4S_0}\right)^{\frac{1}{\gamma+1}}=:\rho_s,\\
&M_0=\frac{m_0}{r_{\rm en}^2\rho_0\sqrt{\gamma S_0\rho_0^{\gamma-1}}}>1,\,\,\mbox{ and }\,\,0<\bar{b}(r)=b_0<\kappa_0\left(\frac{1}{r^4M_0^2}\right)^{\frac{1}{\gamma+1}}.
\end{split}
\end{equation}

\begin{lemma}\label{super-lemma}
For given $\gamma>1$, $r_{\rm en}>1$, $m_0>0$, and $S_0>0$, suppose that constants $\rho_0$, $E_0$, and $M_0$ satisfy \eqref{ini-con-super-1}. 
For a given constant $\bar{\delta}>0$ sufficiently small,
there exists a constant $r_{\ast}\in(r_{\rm en},\infty)$ depending on $r_{\rm en}$, $\gamma$, $\bar{b}$, $m_0$, $S_0$, $\rho_0$, $E_0$, $M_0$, and $\bar{\delta}$ such that 
the initial value problem \eqref{ODE-RE}(or \eqref{ODE-ME}) with \eqref{ini-con-super} 
admits a unique smooth supersonic solution $(\rho, E)(r)\in [C^{\infty}([r_{\rm en}, r_{\ast}])]^2$ 
satisfying
\begin{equation*}
\bar{\delta}\le \rho(r)\le \rho_s-\bar{\delta}\mbox{ for }r\in[r_{\rm en},r_{\ast}]
\end{equation*}
for $\rho_s$ in \eqref{ini-con-super-1}.
%
\end{lemma}
The proof of Lemma \ref{super-lemma} can be obtained by a similar way of \cite[Lemma 1.1]{bae2021structural}, \cite{luo2012transonic}, and \cite[Lemma 1.1]{Duan2023super}. So we skip it.

\begin{remark}\label{rem}
In \cite{Duan2023super}, a parameter set of initial data for \eqref{ini-con-super} was required because the authors computed Sobolev norm estimates in polar coordinates. However, it is not required if we compute the estimates in cartesian coordinates.
\end{remark}

\begin{definition}
For the unique solution $(\bar{\rho},\bar{E})$ in Lemma \ref{super-lemma},
define a background solution ($\bar{\varphi},\bar{\Phi})$ by 
\begin{equation*}
\begin{split}
&\bar{\varphi}({\bf x}):=\int_{r_{\rm en}}^{|{\bf x}|}\frac{m_0}{t^2\bar{\rho}(t)} dt+\frac{m_0}{r_{\rm en}^2\rho_0}\quad\mbox{for}\quad r_{\rm en}<|{\bf x}|<r_{\ast},\\
&\bar{\Phi}({\bf x}):=\int_{r_{\rm en}}^{|{\bf x}|} \bar{E}(t) dt+E_0\quad\mbox{for}\quad r_{\rm en}<|{\bf x}|<r_{\ast}.
\end{split}
\end{equation*}

\end{definition}
%
%
The goal of this work is to  prove the unique existence of the three-dimensional axisymmetric supersonic solution with nonzero vorticity and nonzero angular momentum density to the steady Euler-Poisson system \eqref{E-B} in a divergent nozzle $\Omega$ under some small perturbations of the background solution.

\subsection{Problems and Theorems}
We recall the definition of the divergent nozzle $\Omega$:
\begin{equation*}\label{domain}
\Omega:=\left\{{\bf x}=(x_1,x_2,x_3)\in\mathbb{R}^3:\, 1<r_{\rm en}<r<r_{\rm ex},\,\phi\in[0,\phi_0),\,\theta\in \mathbb{T}\right\}
\end{equation*}
for the spherical coordinates $(r,\phi,\theta)$ of ${\bf x}$ in \eqref{s-coord}.
For notational simplicity,
let us set the entrance $\Gamma_{\rm en}$, wall $\Gamma_{\rm w}$, and the exit $\Gamma_{\rm ex}$ of $\Omega$ as follows:
\begin{equation*}
\Gamma_{\rm en}:=\partial\Omega\cap\{r=r_{\rm en}\},\quad
\Gamma_{\rm w}:=\partial\Omega\cap\{\phi=\phi_0\},\quad
\Gamma_{\rm ex}:=\partial\Omega\cap\{r=r_{\rm ex}\}.
\end{equation*}

Suppose that  ${\bf u},$ $\rho$, $p$, and $\Phi$ are axisymmetric, i.e.,
\begin{equation*}
{\bf u}=u_r(r,\phi){\bf e}_r+u_{\phi}(r,\phi){\bf e}_{\phi}+u_{\theta}(r,\phi){\bf e}_{\theta},\,\, \rho=\rho(r,\phi),\,\, p=p(r,\phi),\,\,\Phi=\Phi(r,\phi).
\end{equation*}
Then, obviously, the entropy $S=p\rho^{-\gamma}$ is also axisymmetric.
Define the angular momentum density $\Lambda$ of the flow by 
\begin{equation}\label{def-Lambda}
\Lambda(r,\phi):=(r\sin\phi)u_{\theta}(r,\phi).
\end{equation}
Then we can rewrite the Euler-Poisson system \eqref{E-B} as a nonlinear system for $(\rho, u_r, u_{\phi}, \Lambda, S,\Phi)$ as follows:
\begin{equation}\label{Axi-sys}
\left\{\begin{split}
&\partial_r(r^2\sin\phi\rho u_r)+\partial_{\phi}(r\sin\phi\rho u_{\phi})=0,\\
&\rho u_r\partial_r u_{\phi}+\frac{\rho u_{\phi}\partial_{\phi}u_{\phi}}{r}+\frac{\partial_{\phi}(S\rho^{\gamma})}{r}+\frac{\rho u_r u_{\phi}}{r}-\frac{\rho}{r}\left(\frac{\Lambda}{r\sin\phi}\right)^2\cot\phi=\frac{\rho\partial_{\phi}\Phi}{r},\\
&(r^2\sin\phi\rho u_r) \partial_rS+(r\sin\phi\rho u_{\phi})\partial_{\phi}S={ 0},\\
&(r^2\sin\phi\rho u_r) \partial_r\Lambda+(r\sin\phi\rho u_{\phi})\partial_{\phi}\Lambda={ 0},\\
&\frac{1}{r^2}\partial_r(r^2\partial_r\Phi)+\frac{1}{r^2\sin\phi}\partial_{\phi}(\sin\phi\partial_{\phi}\Phi)=\rho-b.
\end{split}\right.
\end{equation}

\begin{problem}\label{Pro}
For given  functions $b=b(r,\phi)$ and $$(u_{\rm en},v_{\rm en}, w_{\rm en},S_{\rm en},E_{\rm en},\Phi_{\rm ex})=(u_{\rm en},v_{\rm en}, w_{\rm en},S_{\rm en},E_{\rm en},\Phi_{\rm ex})(\phi),$$
suppose that the following compatibility conditions hold:
\begin{equation}\label{super-comp}
\begin{split}
\partial_{\phi}b=0\mbox{ on }\Gamma_{\rm w},\quad
w_{\rm en}(0)=0,
\quad \frac{d\Phi_{\rm ex}}{d\phi}=0 \mbox{ on }\overline{\Gamma_{\rm ex}}\cap\{\phi=\phi_0\},\\
\frac{du_{\rm en}}{d\phi}=\frac{dv_{\rm en}}{d\phi}=\frac{dw_{\rm en}}{d\phi}
=\frac{dS_{\rm en}}{d\phi}=\frac{dE_{\rm en}}{d\phi}
=0 \mbox{ on }\overline{\Gamma_{\rm en}}\cap\{\phi=\phi_0\}.
\end{split}
\end{equation}
Find $r_{\rm ex}\in(r_{\rm en},r_{\ast})$ for $r_{\ast}$ in Lemma \ref{super-lemma} and an axisymmetric solution $U=(\rho, u_r, u_{\phi}, \Lambda, S,\Phi)$ to \eqref{Axi-sys} in $\Omega$ such that the following properties hold:
\begin{itemize}
\item[(i)] 
In $\overline{\Omega}$, 
	\begin{itemize}
	\item[$\cdot$] $\rho>0$ and $u_r>0$;
	\item[$\cdot$] $|{\bf u}|=\sqrt{u_r^2+u_{\phi}^2+\left(\frac{\Lambda}{r\sin\phi}\right)^2}>c$ for the sound speed $c:=\sqrt{\gamma S\rho^{\gamma-1}}$.
	\end{itemize}
\item[(ii)] $U$ satisfies the following physical boundary conditions:
\begin{equation}\label{sup-bd}
\left\{\begin{split}
(S, \Lambda)=( S_{\rm en},r_{\rm en}\sin\phi w_{\rm en})\,\,&\mbox{on}\,\,\Gamma_{\rm en},\\
(u_r,u_{\phi})=({u}_{\rm en},v_{\rm en}),\,\,\nabla\Phi\cdot{\bf e}_r=E_{\rm en}\,\,&\mbox{on}\,\,\Gamma_{\rm en},\\
u_{\phi}=0,\,\,\nabla\Phi\cdot{\bf e}_{\phi}=0\,\,&\mbox{on}\,\,\Gamma_{\rm w},\\
\nabla\Phi\cdot{\bf e}_r=\Phi_{\rm ex}\,\,&\mbox{on}\,\,\Gamma_{\rm ex}.
\end{split}\right.
\end{equation}
\end{itemize}
\end{problem}

To state our main results, we introduce a Sobolev norm: 
\begin{definition}
For each $k\in\mathbb{N}$ and a function $u:\Omega\to\mathbb{R}$, define a Sobolev norm $\|\cdot\|_{H^k_{\ast}(\Omega)}$ by 
\begin{equation*}
\|u\|_{H^k_{\ast}(\Omega)}:=\|u\|_{H^{k-1}(\Omega)}+\|\partial_ru\|_{H^{k-1}(\Omega)}.
\end{equation*}
For a vector-valued function ${\bf u}=(u_1,\cdots, u_n):\Omega\to\mathbb{R}^n$, we define
\begin{equation*}
\|{\bf u}\|_{H^k_{\ast}(\Omega)}:=\sum_{j=1}^n\|u_j\|_{H^k_{\ast}(\Omega)}.
\end{equation*}
Define $H^k_{\ast}(\Omega)$ and $H^k_{\ast}(\Omega;\mathbb{R}^n)$ by 
\begin{equation*}
\begin{split}
&H^k_{\ast}(\Omega):=\left\{u:\Omega\to\mathbb{R}:\|u\|_{H^k_{\ast}(\Omega)}<\infty\right\},\\
&H^k_{\ast}(\Omega;\mathbb{R}^n):=\left\{{\bf u}=(u_1,\cdots,u_n):\Omega\to\mathbb{R}^n:\sum_{j=1}^n\|u_j\|_{H^k_{\ast}(\Omega)}<\infty\right\}.
\end{split}
\end{equation*}
\end{definition}

Now we state the main theorem.



\begin{theorem}\label{Thm2.2}  Fix $\gamma>1$. 
For given  functions $b=b(r,\phi)\in C^2(\overline{\Omega})$ and 
\begin{equation*}
\begin{split}
(u_{\rm en},v_{\rm en}, w_{\rm en},S_{\rm en},E_{\rm en},\Phi_{\rm ex})
&=(u_{\rm en},v_{\rm en}, w_{\rm en},S_{\rm en},E_{\rm en},\Phi_{\rm ex})(\phi)\\
&\in C^3(\overline{\Gamma_{\rm en}})\times [C^4(\overline{\Gamma_{\rm en}})]^4\times C^4(\overline{\Gamma_{\rm ex}}),
\end{split}
\end{equation*}
suppose that the compatibility conditions in \eqref{super-comp} hold.
For simplicity of notations, let us set ${\bm\tau}_p$ as
\begin{equation*}
\begin{split}
{\bm\tau}_p:=&\|b-\bar{b}\|_{C^2(\overline{\Omega})}+\|u_{\rm en}-\frac{m_0}{r_{\rm en}^2\rho_0}\|_{C^3(\overline{\Gamma_{\rm en}})}+\|\Phi_{\rm ex}-\bar{\Phi}\|_{C^4(\overline{\Gamma_{\rm ex}})}
\\
&+\|v_{\rm en}\|_{C^4(\overline{\Gamma_{\rm en}})}
+\|w_{\rm en}\|_{C^4(\overline{\Gamma_{\rm en}})}
+\|S_{\rm en}-S_0\|_{C^4(\overline{\Gamma_{\rm en}})}
+\|E_{\rm en}-E_0\|_{C^4(\overline{\Gamma_{\rm en}})}.
\end{split}
\end{equation*}
There exist $r_{\rm ex}\in(r_{\rm en},r_{\ast})$ for $r_{\ast}$ in Lemma \ref{super-lemma} 
and a small constant ${\bm\tau}^{\star}>0$ depending only on $\gamma$, $\bar{b}$,  $m_0$, $\rho_0$, $\bar{\varphi}$, $\bar{\rho}$, $\bar{\Phi}$, $S_0,$ $E_0$,  $r_{\rm en}$, $r_{\rm ex},$ $\bar{\delta}$, and $\phi_0$ so that if 
\begin{equation*}
{\bm\tau}_p\le{\bm\tau}^{\star},
\end{equation*}
then Problem \ref{Pro} has a unique axisymmetric  solution $({\bf u},\rho, S,\Phi)$ that satisfies the estimate 
\begin{equation}\label{super-est-thm}
\|{\bf u}-\nabla\bar{\varphi}\|_{H_{\ast}^3(\Omega;\mathbb{R}^3)}
+\|\rho-\bar{\rho}\|_{H_{\ast}^3(\Omega)}
+\|S-S_0\|_{H_{\ast}^4(\Omega)}
+\|\Phi-\bar{\Phi}\|_{H^4_{\ast}(\Omega)}\le C{\bm\tau}_p
\end{equation}
for a constant $C>0$ depending only on $\gamma$, $\bar{b}$,  $m_0$, $\rho_0$, $\bar{\varphi}$, $\bar{\rho}$, $\bar{\Phi}$, $S_0,$ $E_0$,  $r_{\rm en}$, $r_{\rm ex},$ $\bar{\delta}$, and $\phi_0$.
\end{theorem}
\begin{remark}
Unlike \cite{bae2025supersonic}, we do not use the weighted Sobolev norms.
In \cite{bae2025supersonic}, authors used the method of reflection across the entrance and the weighted Sobolev norms with a distance function from the exit boundary to resolve the corner singularity issues. 
In this paper, however, we use the method of reflection across the wall and the standard Sobolev norms.
\end{remark}



\section{Reformulation Problems via Helmholtz decomposition}\label{Sec-HD}

We decompose the velocity vector field ${\bf u}$ as
\begin{equation}\label{HD}
{\bf u}=\nabla\varphi+\mbox{curl}{\bf V}
\end{equation}
for axisymmetric functions
\begin{equation*}
\varphi=\varphi(r,\phi)\mbox{ and }{\bf V}=h_r(r,\phi){\bf e}_r+h_{\phi}(r,\phi){\bf e}_{\phi}+\psi(r,\phi){\bf e}_{\theta}.
\end{equation*}
If $\varphi$, $h_r$, $h_{\phi}$, and $\psi$ are $C^2$, then  a direct computation implies 
\begin{equation*}
{\bf u}=u_r{\bf e}_r+u_{\phi}{\bf e}_{\phi}+u_{\theta}{\bf e}_{\theta}.
\end{equation*}
with
\begin{equation*}
\begin{split}
&u_r:=\partial_r\varphi+\frac{1}{r\sin\phi}\partial_{\phi}( \psi\sin\phi)=\partial_r\varphi+[\nabla\times(\psi{\bf e}_{\theta})]\cdot{\bf e}_r,\\
&u_{\phi}:=\frac{1}{r}(\partial_{\phi}\varphi)-\frac{1}{r}\partial_r(r\psi)=\frac{1}{r}(\partial_{\phi}\varphi)+[\nabla\times(\psi{\bf e}_{\theta})]\cdot{\bf e}_\phi,\\
&u_{\theta}:=\frac{1}{r}\left(\partial_r(rh_{\phi})-\partial_{\phi}h_r\right)=\frac{\Lambda}{r\sin\phi}.
\end{split}
\end{equation*}
Since
$u_{\theta}=\frac{\Lambda}{r\sin\phi}$
by the definition of $\Lambda$ in \eqref{def-Lambda} and we will find $\Lambda$, not $h_{r}$ and $h_{\phi}$,
we set
\begin{equation}
{\bf q}\left(\nabla\varphi,\nabla\times(\psi{\bf e}_{\theta}),\frac{\Lambda}{r\sin\phi}\right):=
{\bf u}
\end{equation}
For notational simplicity, let us set 
\begin{equation*}
{\bf t}\left(\nabla\times(\psi{\bf e}_{\theta}),\frac{\Lambda}{r\sin\phi}\right):={\bf q}\left(\nabla\varphi,\nabla\times(\psi{\bf e}_{\theta}),\frac{\Lambda}{r\sin\phi}\right)-\nabla\varphi\,(=\mbox{curl}{\bf V}).
\end{equation*}
Then we can rewrite the Euler-Poisson system \eqref{Axi-sys} in terms of $(\varphi,\Phi,\psi{\bf e}_{\theta},S,\Lambda)$ as follows:
\begin{equation}\label{HD-sys}
\left\{\begin{split}
&\mbox{div}\left(\varrho(S,\Phi,{\bf q}){\bf q}\right)=0,\\
&\Delta\Phi=\varrho(S,\Phi,{\bf q})-b,\\
&-\Delta(\psi{\bf e}_\theta)=G(S,\Lambda,\partial_{\phi}S,\partial_{\phi}\Lambda, \Phi,{\bf t},\nabla\varphi){\bf e}_{\theta},\\
&\varrho(S,\Phi,{\bf q}){\bf q}\cdot\nabla(S,\Lambda)={\bf 0}
\end{split}\right.
\end{equation}
with ${\bf q}$, ${\bf t}$, $\varrho$ and $G$ defined by 
\begin{equation}\label{def-H-G}
\left.\begin{split}
&{\bf q}:={\bf q}\left(\nabla\varphi,\nabla\times(\psi{\bf e}_{\theta}),\frac{\Lambda}{r\sin\phi}\right),\quad{\bf t}:={\bf t}\left(\nabla\times(\psi{\bf e}_{\theta}),\frac{\Lambda}{r\sin\phi}\right),\\
&\varrho(\eta,z,{\bf p}):=\left[\frac{\gamma-1}{\gamma\eta}\left(z-\frac{1}{2}|{\bf p}|^2\right)\right]^{\frac{1}{\gamma-1}},\\
&G(\eta_1,\eta_2,\eta_3,\eta_4,z,{\bf s},{\bf v}):=\frac{1}{r({\bf s}+{\bf v})\cdot{\bf e}_r}\left(\frac{\eta_3\varrho^{\gamma-1}(\eta_1, z,{\bf s}+{\bf v})}{\gamma-1}+\frac{\eta_2\eta_4}{r^2\sin^2\phi}\right)
\end{split}
\right.
\end{equation}
for $\eta\in\mathbb{R}$, $z\in\mathbb{R}$, ${\bf p}\in\mathbb{R}^3$, $\eta_1,\eta_2,\eta_3,\eta_4\in\mathbb{R}$, and ${\bf s},{\bf v}\in\mathbb{R}^3$.

One can easily check that if the following boundary conditions for 
 $(\varphi,\Phi, \psi{\bf e}_{\theta},S,\Lambda)$ hold, then the boundary conditions in \eqref{sup-bd} hold:
\begin{equation}\label{HD-bd-sup}
\left\{\begin{split}
(S,\Lambda)=( S_{\rm en},r\sin\phi w_{\rm en})
\,\,&\mbox{on}\,\,\Gamma_{\rm en},\\
\varphi=\varphi_{\rm en},\,\,\nabla\varphi\cdot{\bf e}_r=u_{\rm en}-[\nabla\times(\psi{\bf e}_{\theta})]\cdot{\bf e}_r=u_{\rm en}-\frac{\partial_{\phi}(\psi\sin\phi)}{r\sin\phi}\,\, 
&\mbox{on}\,\,\Gamma_{\rm en},\\
[\nabla\times(\psi{\bf e}_{\theta})]\cdot{\bf e}_\phi=0,\,\,\nabla\Phi\cdot{\bf e}_r=E_{\rm en}\,\,&\mbox{on}\,\,\Gamma_{\rm en},\\
\nabla\varphi\cdot{\bf e}_{\phi}=0,\,\,[\nabla\times(\psi{\bf e}_{\theta})]\cdot{\bf e}_\phi={ 0},\,\,
\nabla\Phi\cdot{\bf e}_{\phi}=0\,\,&\mbox{on}\,\,\Gamma_{\rm w},\\
\psi{\bf e}_{\theta}={\bf 0},\,\,\nabla\Phi\cdot{\bf e}_r=\Phi_{\rm ex}\,\,&\mbox{on}\,\,\Gamma_{\rm ex}.
\end{split}\right.
\end{equation}
In \eqref{HD-bd-sup}, $\varphi_{\rm en}=\varphi_{\rm en}(\phi)$ is a function defined by 
\begin{equation*}
\varphi_{\rm en}(\phi):=r_{\rm en}\int_0^{\phi}v_{\rm en}(t)dt+\bar{\varphi}(r_{\rm en}).
\end{equation*} 
Also, if it holds that $\nabla(\psi{\bf e}_{\theta})\cdot{\bf e}_r+\frac{\psi{\bf e}_{\theta}}{r}={\bf 0}$ on $\Gamma_{\rm en}$ and $\psi{\bf e}_{\theta}={\bf 0}$ on $\Gamma_{\rm w}$, then we have $[\nabla\times(\psi{\bf e}_{\theta})]\cdot{\bf e}_\phi={ 0}$ on $\Gamma_{\rm en}\cup \Gamma_{\rm w}$.
Thus we will find $\psi{\bf e}_{\theta}$ satisfying 
\begin{equation*}
\nabla(\psi{\bf e}_{\theta})\cdot{\bf e}_r+\frac{\psi{\bf e}_{\theta}}{r}={\bf 0} \mbox{ on }\Gamma_{\rm en}\mbox{ and }\,\,
\psi{\bf e}_{\theta}={\bf 0}\mbox{ on }\Gamma_{\rm w}.
\end{equation*}

Now we state the main theorem for the reformulated problem in terms of $(\varphi,\Phi,\psi,S,\Lambda)$:

\begin{theorem}\label{M-Thm-super}
Under the same assumptions of Theorem \ref{Thm2.2},
there exist $r_{\rm ex}\in(r_{\rm en},r_{\ast})$ for $r_{\ast}$ in Lemma \ref{super-lemma}  and a small constant ${\bm\tau}_{\star}>0$ depending only on $\gamma$, $\bar{b}$,  $m_0$, $\rho_0$, $\bar{\varphi}$, $\bar{\rho}$, $\bar{\Phi}$, $S_0,$ $E_0$,  $r_{\rm en}$, $r_{\rm ex},$ $\bar{\delta}$, and $\phi_0$ so that if 
\begin{equation*}
{\bm \tau}_p\le{\bm\tau}_{\star},
\end{equation*}
then the boundary value problem \eqref{HD-sys} with  \eqref{HD-bd-sup} has a unique axisymmetric solution $(\varphi,\Phi,\psi{\bf e}_{\theta},S,\Lambda)$ that satisfies the estimate 
\begin{equation}\label{est-super-thm}
\begin{split}
\|\varphi-\bar{\varphi}\|_{H_{\ast}^4(\Omega)}&+\|\Phi-\bar{\Phi}\|_{H_{\ast}^4(\Omega)}+\|\psi{\bf e}_{\theta}\|_{H_{\ast}^5(\Omega;\mathbb{R}^3)}\\
&+\|S-S_0\|_{H_{\ast}^4(\Omega)}+\left\|\frac{\Lambda}{r\sin\phi}{\bf e}_{\theta}\right\|_{H_{\ast}^3(\Omega;\mathbb{R}^3)}\le C{\bm\tau}_p
\end{split}
\end{equation}
for a constant $C>0$ depending only on $\gamma$, $\bar{b}$,  $m_0$, $\rho_0$, $\bar{\varphi}$, $\bar{\rho}$, $\bar{\Phi}$, $S_0,$ $E_0$,  $r_{\rm en}$, $r_{\rm ex},$ $\bar{\delta}$, and $\phi_0$.
\end{theorem}

Hereafter, a constant $C$ is said to be chosen depending only on the data if $C$ is chosen depending only on $\gamma$, $\bar{b}$,  $m_0$, $\rho_0$, $\bar{\varphi}$, $\bar{\rho}$, $\bar{\Phi}$, $S_0,$ $E_0$,  $r_{\rm en}$,  $r_{\ast}$, $\bar{\delta}$, and $\phi_0$.

Once we get \eqref{est-super-thm}, we can see that the solution is a classical solution by applying the standard elliptic theory for the interior regularity to \eqref{W-super}, \eqref{C-N-lin}, and \eqref{lin-1}.
One can easily prove that Theorem  \ref{M-Thm-super} implies Theorem \ref{Thm2.2}.
Thus the rest of the paper is devoted to proving Theorem  \ref{M-Thm-super}. 

\section{Proof of Theorem \ref{M-Thm-super}}\label{last-sec}
For positive constants $(\delta_1,\delta_2,\delta_3,L:=r_{\rm ex}-r_{\rm en})$ to be determined later, define iteration sets as follows:

\begin{itemize}
\item[(i)] Iteration set for $(S,\Lambda)$: 
Define an iteration set $\mathcal{J}(\delta_1)$ by 
\begin{equation}\label{J1}
\mathcal{J}(\delta_1):=\mathcal{J}_1(\delta_1)\times\mathcal{J}_2(\delta_1)
\end{equation}
with 
\begin{equation*}
\begin{split}
&\mathcal{J}_1(\delta_1):=\left\{S=S(r,\phi)\in H_{\ast}^4(\Omega)\left|\,
	\begin{split}
	& \|S-S_0\|_{H_{\ast}^4(\Omega)}\le \delta_1{\bm\tau}_p,\\
	& S=S_{\rm en}\mbox{ on }\Gamma_{\rm en},\\
	&\partial_{\phi} S=0\mbox{ on }\Gamma_{\rm w}
	\end{split}\right.\right\},\\
&\mathcal{J}_2(\delta_1):=\left\{\mathcal{V}{\bf e}_{\theta}=\frac{\Lambda(r,\phi)}{r\sin\phi}{\bf e}_{\theta}\in [H_{\ast}^3(\Omega)]^3\left|\, 
	\begin{split}
	&\left\|\mathcal{V}{\bf e}_{\theta}\right\|_{H_{\ast}^3(\Omega;\mathbb{R}^3)}\le \delta_1{\bm\tau}_p,\,\\
	&\mathcal{V}= w_{\rm en}\mbox{ on }\Gamma_{\rm en},\\
	&\partial_{\phi}\Lambda=0\mbox{ on }\Gamma_{\rm w}
	\end{split}\right.\right\}.
\end{split}
\end{equation*}

\item[(ii)] Iteration set for $(\chi,\Psi)$: Define an iteration set $\mathcal{J}(\delta_2)$ by 
\begin{equation}\label{J2}
\mathcal{J}(\delta_2):=\mathcal{K}(\delta_2)\times\mathcal{K}(\delta_2)
\end{equation}
with
\begin{equation*}
\mathcal{K}(\delta_2):=\left\{\zeta=\zeta(r,\phi)\in H_{\ast}^4(\Omega)\left|\,
	\begin{split}
	&\|\zeta\|_{H_{\ast}^4(\Omega)}\le \delta_2{\bm\tau}_p,\\
	&\partial_{\phi}\zeta=0\mbox{ on }\Gamma_{\rm w}
	\end{split}\right.
	 \right\}.
\end{equation*}

\item[(iii)] Iteration set for $\psi{\bf e}_{\theta}$: Define an iteration set $\mathcal{J}(\delta_3)$ by 
\begin{equation}\label{J3}
\begin{split}
&\mathcal{J}(\delta_3):=\left\{{\bf W}=\psi(r,\phi){\bf e}_{\theta}\in [H_{\ast}^5(\Omega)]^3\left|\,
	\begin{split}
	&\|{\bf W}\|_{H_{\ast}^5(\Omega;\mathbb{R}^3)}\le \delta_3{\bm\tau}_p,\\
	&{\bf W}={\bf 0}\mbox{ on }\Gamma_{\rm w},\\
	&\partial_\phi(\nabla\times{\bf W}\cdot{\bf e}_r)=0\mbox{ on }\Gamma_{\rm w}
	\end{split}\right.\right\}.
\end{split}
\end{equation}
\end{itemize}
For a fixed $(S_{\ast},\frac{\Lambda_{\ast}}{r\sin\phi}{\bf e}_{\theta})\in\mathcal{J}(\delta_1)$, we first solve the following boundary value problem for $(\varphi,\Phi,\psi{\bf e}_{\theta})$:
\begin{equation}\label{super-sys}
\begin{split}
&\left\{\begin{split}
	&\mbox{div}\left(\varrho(S_{\ast},\Phi,{\bf q}_{\ast}){\bf q}_{\ast}\right)=0\\
	&\Delta\Phi=\varrho(S_{\ast},\Phi,{\bf q}_{\ast})-b\\
	&-\Delta(\psi{\bf e}_\theta)=G(S_{\ast},\Lambda_{\ast},\partial_{\phi}S_{\ast},\partial_{\phi}	\Lambda_{\ast},\Phi,{\bf t}_{\ast},\nabla\varphi){\bf e}_{\theta}\end{split}\right.\quad\mbox{in}\,\,\Omega,\\
&\left\{\begin{split}
	\varphi=\varphi_{\rm en},\,\,\nabla\varphi\cdot{\bf e}_r=u_{\rm en}-[\nabla\times(\psi{\bf e}_{\theta})]\cdot{\bf e}_r,\,\,\nabla\Phi\cdot{\bf e}_r=E_{\rm en}\,\, 
	&\mbox{on}\,\,\Gamma_{\rm en},\\
	\nabla(\psi{\bf e}_{\theta})\cdot{\bf e}_r+\frac{\psi{\bf e}_{\theta}}{r}={\bf 0}\,\,&\mbox{on}\,\,\Gamma_{\rm en},\\
	\nabla\varphi\cdot{\bf e}_{\phi}=0,\,\,
\nabla\Phi\cdot{\bf e}_{\phi}=0,\,\, \psi{\bf e}_{\theta}={\bf 0}\,\,&\mbox{on}\,\,\Gamma_{\rm w},\\
	\nabla\Phi\cdot{\bf e}_r=\Phi_{\rm ex},\,\,\psi{\bf e}_{\theta}={\bf 0}\,\,&\mbox{on}\,\,\Gamma_{\rm ex}
	\end{split}\right.
\end{split}
\end{equation}
with 
\begin{equation*}
{\bf q}_{\ast}:={\bf q}\left(\nabla\varphi,\nabla\times(\psi{\bf e}_{\theta}),\frac{\Lambda_{\ast}}{r\sin\phi}\right)\mbox{ and }{\bf t}_{\ast}:={\bf t}\left(\nabla\times(\psi{\bf e}_{\theta}),\frac{\Lambda_{\ast}}{r\sin\phi}\right).
\end{equation*}


\begin{proposition}\label{su-lemma-1}
For $r_{\ast}>0$ in Lemma \ref{super-lemma}, 
there exist a constant $L(=r_{\rm ex}-r_{\rm en})\in(0,r_{\ast}-r_{\rm en}]$  depending only on the data and
 a small constant ${\bm\tau}_1>0$ depending only on the data and $(L, \delta_1,\delta_2,\delta_3)$ so that if 
\begin{equation*}
{\bm\tau}_p\le{\bm\tau}_1,
\end{equation*}
then the boundary value problem \eqref{super-sys} has a unique axisymmetric solution $(\varphi,\Phi,\psi{\bf e}_{\theta})$ that satisfies \begin{equation}\label{lem1-est}
\|\varphi-\bar{\varphi}\|_{H_{\ast}^4(\Omega)}+\|\Phi-\bar{\Phi}\|_{H_{\ast}^4(\Omega)}+\|\psi{\bf e}_{\theta}\|_{H_{\ast}^5(\Omega;\mathbb{R}^3)}\le C(1+\delta_1){\bm\tau}_p
\end{equation}
for a constant $C>0$ depending only on the data and $L$.
\end{proposition}


\begin{proof}[Proof of Proposition \ref{su-lemma-1}]
For a fixed ${\bf W}_{\ast}=\psi_{\ast}{\bf e}_{\theta}\in\mathcal{J}(\delta_3)$, let us consider the following nonlinear boundary problem for $(\varphi,\Phi)$:
\begin{equation}\label{pp-prob-sup}
\left\{\begin{split}
\mbox{div}\left(\varrho(S_{\ast},\Phi,{\bf q}^{\ast}){\bf q}^{\ast}\right)=0,\,\,\Delta\Phi=\varrho(S_{\ast},\Phi,{\bf q}^{\ast})-b\,\,&\mbox{in}\,\,\Omega,\\
\varphi=\varphi_{\rm en},\,\,\nabla\varphi\cdot{\bf e}_r=u_{\rm en}-[\nabla\times(\psi_{\ast}{\bf e}_{\theta})]\cdot{\bf e}_r,\,\,\nabla\Phi\cdot{\bf e}_r=E_{\rm en}\,\,&\mbox{on}\,\,\Gamma_{\rm en},\\
\partial_{\phi}\varphi=0,\,\,
\nabla\Phi\cdot{\bf e}_{\phi}=0\,\,&\mbox{on}\,\,\Gamma_{\rm w},\\
\nabla\Phi\cdot{\bf e}_r=\Phi_{\rm ex}\,\,&\mbox{on}\,\,\Gamma_{\rm ex}
\end{split}\right.
\end{equation}
with
\begin{equation}\label{t-ast}
{\bf q}^{\ast}:={\bf q}\left(\nabla\varphi,\nabla\times(\psi_{\ast}{\bf e}_{\theta}),\frac{\Lambda_{\ast}}{r\sin\phi}\right)\mbox{ and }{\bf t}^{\ast}:={\bf t}\left(\nabla\times(\psi_{\ast}{\bf e}_{\theta}),\frac{\Lambda_{\ast}}{r\sin\phi}\right).
\end{equation}


\begin{lemma}\label{lema-po}
For $r_{\ast}>0$ in Lemma \ref{super-lemma}, 
there exist a constant $L(=r_{\rm ex}-r_{\rm en})\in(0,r_{\ast}-r_{\rm en}]$ depending only on the data and
a small constant ${\bm\tau}_2>0$ depending only on the data  and $(L, \delta_1,\delta_2,\delta_3)$ so that if 
$${\bm\tau}_p\le {\bm\tau}_2,$$
then the boundary value problem \eqref{pp-prob-sup} has a unique axisymmetric solution $(\varphi,\Phi)$ that satisfies the estimate
\begin{equation}\label{po=est}
\|\varphi-\bar{\varphi}\|_{H_{\ast}^4(\Omega)}+\|\Phi-\bar{\Phi}\|_{H_{\ast}^4(\Omega)}\le C\left(1+\delta_1+\delta_3\right){\bm\tau}_p
\end{equation}
for a constant $C>0$ depending only on the data and $L$.
\end{lemma}
The proof of Lemma \ref{lema-po} is given in the next section.
%
For  the solution $(\varphi,\Phi)$ obtained in Lemma \ref{lema-po},
let us set 
$$G_{\sharp}:=G(S_{\ast},\Lambda_{\ast},\partial_{\phi}S_{\ast},\partial_{\phi}\Lambda_{\ast},\Phi,{\bf t}_{\ast},\nabla\varphi)$$ for notational simplicity.
We solve  the following boundary value problem for ${\bf W}$:
\begin{equation}\label{W-super}
\left\{\begin{split}
-\Delta{\bf W}=G_{\sharp}{\bf e}_{\theta}\quad&\mbox{in}\,\,\Omega,\\
\nabla{\bf W}\cdot{\bf e}_r+\frac{{\bf W}}{r}={\bf 0}\quad&\mbox{on}\,\,\Gamma_{\rm en},\\
{\bf W}={\bf 0}\quad&\mbox{on}\,\,\Gamma_{\rm w}\cup\Gamma_{\rm ex}.
\end{split}\right.
\end{equation}


\begin{lemma}\label{lema-vor}
There exists a small constant ${\bm\tau}_3\in(0,{\bm\tau}_2]$ depending only on the data and $(L, \delta_1,\delta_2,\delta_3)$ so that if 
$${\bm\tau}_p\le {\bm\tau}_3,$$
then the boundary value problem \eqref{W-super} has a unique solution ${\bf W}$ of the form 
\begin{equation*}
{\bf W}=\psi{\bf e}_{\theta}
\end{equation*}
for an axisymmetric function $\psi=\psi(r,\phi)$ satisfying 
\begin{equation*}
\left\{\begin{split}
-\left(\Delta_{\bf x}-\frac{1}{r^2\sin^2\phi}\right)\psi=G_{\sharp}\quad&\mbox{in}\,\,\Omega,\\
\frac{1}{r}\nabla(r\psi)\cdot{\bf e}_r=0\quad&\mbox{on}\,\,\Gamma_{\rm en},\\
\psi=0\quad&\mbox{on}\,\,\Gamma_{\rm w}\cup\Gamma_{\rm ex}\cup[{\Omega}\cap\{\phi=0\}].\\
\end{split}\right.
\end{equation*}
Furthermore, the solution satisfies the estimate 
\begin{equation}\label{W-est}
\|{\bf W}\|_{H_{\ast}^5(\Omega;\mathbb{R}^3)}\le C\|G_{\sharp}{\bf e}_{\theta}\|_{H_{\ast}^3(\Omega;\mathbb{R}^3)}\le C_{\dagger}\delta_1{\bm\tau}_p
\end{equation}
for constants $C, C_{\dagger}>0$ depending only on the data and $L$.
\end{lemma}
By the standard elliptic theory, one can prove 
the unique existence of a solution ${\bf W}$ to \eqref{W-super} and the estimate \eqref{W-est}.
To prove that  ${\bf W}$ has of the form ${\bf W}=\psi{\bf e}_{\theta}$ for an axisymmetric function $\psi$,  one can adjust and apply the proof of Theorem 3.1 of \cite{park2020transonic}. So we skip it. 
%
For the unique solution ${\bf W}$ to \eqref{W-super} associated with ${\bf W}_{\ast}\in \mathcal{J}(\delta_3)$,
we define an iteration mapping $\mathcal{I}_v:\mathcal{J}(\delta_3)\to H_{\ast}^5(\Omega;\mathbb{R}^3)$ by 
\begin{equation}\label{mapping-W}
\mathcal{I}_v({\bf W}_{\ast})={\bf W}.
\end{equation}
If we choose $\delta_3$ as 
\begin{equation*}
\delta_3:=C_{\dagger}\delta_1
\end{equation*}
for a constant $C_{\dagger}$ in \eqref{W-est}, then the mapping $\mathcal{I}_v$ maps $\mathcal{J}(\delta_3)$ into itself.
By applying the Schauder fixed point theorem, one can prove that the mapping $\mathcal{I}_v$ has a fixed point ${\bf W}_{\sharp}:=\psi_{\sharp}{\bf e}_{\theta}\in\mathcal{J}(\delta_3)$.
For the solution $(\varphi_{\sharp},\Phi_{\sharp})$ to \eqref{pp-prob-sup} associated with ${\bf W}_{\ast}={\bf W}_{\sharp}$, it is clear that $(\varphi_{\sharp},\Phi_{\sharp},{\bf W}_{\sharp})$ is a solution to the boundary value problem \eqref{super-sys}, and it satisfies the estimate \eqref{lem1-est}.

To complete the proof of Proposition \ref{su-lemma-1}, it remains to prove the uniqueness of a solution. 
Suppose that there exist two solutions $(\varphi_1,\Phi_1,{\bf W}_1)$ and $(\varphi_2,\Phi_2,{\bf W}_2)$.
Then ${\bf W}_1-{\bf W}_2$ satisfies 
\begin{equation*}
\left\{\begin{split}
-\Delta({\bf W}_1-{\bf W}_2)=(G_1-G_2){\bf e}_{\theta}\quad&\mbox{in}\,\,\Omega,\\
\nabla({\bf W}_1-{\bf W}_2)\cdot{\bf e}_r+\frac{{\bf W}_1-{\bf W}_2}{r}={\bf 0}\quad&\mbox{on}\,\,\Gamma_{\rm en},\\
{\bf W}_1-{\bf W}_2={\bf 0}\quad&\mbox{on}\,\,\Gamma_{\rm w}\cup\Gamma_{\rm ex}
\end{split}\right.
\end{equation*}
for $G_k$ ($k=1,2$) defined by 
\begin{equation*}
G_k:=G(S_{\ast},\Lambda_{\ast},\partial_{\phi}S_{\ast},\partial_{\phi}\Lambda_{\ast},\Phi_k,{\bf t}_{k},\nabla\varphi_k)
\mbox{ with }{\bf t}_k:={\bf t}\left(\nabla\times{\bf W}_k,\frac{\Lambda_{\ast}}{r\sin\phi}\right).
\end{equation*}
Also, for ${\bm\tau}_3$ in Lemma \ref{lema-vor}, there exists a sufficiently small constant ${\bm\tau}_{\diamond}\in(0,{\bm\tau}_3]$ depending only on the data and $(L,\delta_1,\delta_2,\delta_3)$ so that if ${\bm\tau}_p\in(0,{\bm\tau}_{\diamond}]$, then  ${\bf W}_1-{\bf W}_2$ satisfies
\begin{equation}\label{W12-con}
\begin{split}
\|{\bf W}_1-{\bf W}_2\|_{H^2(\Omega;\mathbb{R}^3)}
&\le C_{\ast}\|(G_1-G_2){\bf e}_{\theta}\|_{L^2(\Omega;\mathbb{R}^3)}\\
&\le C^{\ast}{\bm\tau}_p\|{\bf W}_1-{\bf W}_1\|_{H^2(\Omega;\mathbb{R}^3)}
\end{split}
\end{equation}
for positive constants $C_{\ast}$ and $C^{\ast}$ depending only on the data and $L$.
Therefore,
we can choose  ${\bm\tau}_1\in(0,{\bm\tau}_{\diamond}]$ depending only on the data and $(L,\delta_1,\delta_2,\delta_3)$ so that \eqref{W12-con} implies ${\bf W}_1={\bf W}_2$.
This finishes the proof of Proposition \ref{su-lemma-1}.
\end{proof}

Let $(\varphi_{\ast},\Phi_{\ast},\psi_{\ast}{\bf e}_{\theta})$ be the axisymmetric solution to the boundary value problem \eqref{super-sys} associated with $(S_{\ast},\frac{\Lambda_{\ast}}{r\sin\phi}{\bf e}_{\theta})$. 
For notational simplicity, we set
\begin{equation*}
{\bf q}^{\star}:={\bf q}\left(\nabla\varphi_{\ast},\nabla\times(\psi_{\ast}{\bf e}_{\theta}),\frac{\Lambda_{\ast}}{r\sin\phi}\right)
\end{equation*}
and consider the following initial value problem for $(S,\Lambda)$:
\begin{equation}\label{INI_pro-super}
\left\{\begin{split}
\varrho(S_{\ast},\Phi_{\ast},{\bf q}^{\star}){\bf q}^{\star}\cdot\nabla(S,\Lambda)={\bf 0}\quad&\mbox{in}\,\,\Omega,\\
(S,\Lambda)=( S_{\rm en},r\sin\phi\, w_{\rm en})\quad&\mbox{on}\,\,\Gamma_{\rm en}.
\end{split}\right.
\end{equation}

\begin{lemma}\label{super-lemma-2}
Let ${\bm\tau}_1$ be from Proposition \ref{su-lemma-1}.
There exists a small constant ${\bm\tau}_4\in(0,{\bm\tau}_1]$ depending only on the data and $(L,\delta_1,\delta_2,\delta_3)$ so that if 
\begin{equation*}
{\bm\tau}_p\le{\bm\tau}_4,
\end{equation*}
then the initial value problem \eqref{INI_pro-super} has a unique axisymmetric solution $(S,\Lambda)$ that satisfies
\begin{equation}\label{super-est-trans}
\|S-S_0\|_{H_{\ast}^4(\Omega)}+\left\|\frac{\Lambda}{r\sin\phi}{\bf e}_{\theta}\right\|_{H_{\ast}^3(\Omega;\mathbb{R}^3)}\le C_{\diamond}{\bm\tau}_p
\end{equation}
for a constant $C_{\diamond}>0$ depending only on the data and $L$. 
\end{lemma}
Since $S$ and $\Lambda$ are conserved along each streamline, one can find the solution to \eqref{INI_pro-super} by using the stream function for axisymmetric flows.
We skip the proof of Lemma \ref{super-lemma-2} because it can be verified by a similar way of  \cite[Proof of Lemma 4.3]{park2020transonic}.

For the axisymmetric solution $(\varphi_{\ast},\Phi_{\ast},\psi_{\ast}{\bf e}_{\theta})$ to the boundary value problem \eqref{super-sys} associated with $(S_{\ast},\frac{\Lambda_{\ast}}{r\sin\phi}{\bf e}_{\theta})$, let $(S,\Lambda)$ be the axisymmetric solution to the initial value problem \eqref{INI_pro-super} associated with $(\varphi_{\ast},\Phi_{\ast},\psi_{\ast}{\bf e}_{\theta})$.
We define an iteration mapping $\mathcal{I}_t:\mathcal{J}(\delta_1)\to H_{\ast}^4({\Omega})\times H_{\ast}^3(\Omega;\mathbb{R}^3)$ by 
\begin{equation*}
\mathcal{I}_t\left(S_{\ast},\frac{\Lambda_{\ast}}{r\sin\phi}{\bf e}_{\theta}\right)=\left(S,\frac{\Lambda}{r\sin\phi}{\bf e}_{\theta}\right),
\end{equation*}
and  choose $\delta_1$ as 
\begin{equation*}
\delta_1:=C_{\diamond}
\end{equation*}
for the constant $C_{\diamond}>0$ in \eqref{super-est-trans} so that 
the mapping $\mathcal{I}_t$ maps $\mathcal{J}(\delta_1)$ into itself.
Applying the Schauder fixed point theorem gives the mapping $\mathcal{I}_t$ has a fixed point $(S_{\sharp},\frac{\Lambda_{\sharp}}{r\sin\phi}{\bf e}_{\theta})\in\mathcal{J}(\delta_1)$. 
Then, for the axisymmetric solution $(\varphi_{\sharp},\Phi_{\sharp},\psi_{\sharp}{\bf e}_{\theta})$  to \eqref{super-sys} associated with $(S_{\sharp},\frac{\Lambda_{\sharp}}{r\sin\phi}{\bf e}_{\theta})$, it is clear that $(\varphi_{\sharp},\Phi_{\sharp},\psi_{\sharp}{\bf e}_{\theta},S_{\sharp},\Lambda_{\sharp})$  
is an axisymmetric solution to the boundary value problem \eqref{HD-sys} with  \eqref{HD-bd-sup} and satisfies the estimate \eqref{est-super-thm}.
%
By using the standard contraction argument, one can prove that the solution is unique if ${\bm\tau}_p\le{\bm\tau}_\star$   for a sufficiently small constant ${\bm\tau}_\star\in(0,{\bm\tau}_4]$ depending only on the data and $(L,\delta_1,\delta_2,\delta_3)$.
This finishes the proof of Theorem \ref{M-Thm-super}.\qed 

\section{Proof of Lemma \ref{lema-po}}\label{sec-lem-last} 
To complete the proof of the main theorem, it remains to prove Lemma \ref{lema-po}.
%
Set $\chi:=\varphi-\bar{\varphi}$ and $\Psi:=\Phi-\bar{\Phi}$ to linearize 
\begin{equation}\label{C-N-lin}
\left\{\begin{split}
&\mbox{div}\left(\varrho(S,\Phi,{\bf q}){\bf q}\right)=0,\\
&\Delta\Phi=\varrho(S,\Phi,{\bf q})-b,
\end{split}\right.
\end{equation}
and set 
\begin{equation*}
\begin{split}
&(\partial_1,\partial_2):=(\partial_r,\partial_\phi),\quad D:=(\partial_r,\frac{1}{r}\partial_{\phi},\frac{1}{r\sin\phi}\partial_{\theta}),\\
&(q_1,q_2):=({\bf q}\cdot{\bf e}_r,{\bf q}\cdot{\bf e}_{\phi}),\quad 
c^2(\Phi,{\bf q}):=(\gamma-1)\left(\Phi-\frac{1}{2}|{\bf q}|^2\right)
\end{split}
\end{equation*}
for simplicity of notations.
Then, by a direct computation, we can rewrite the first equation $\mbox{div}\left(\varrho(S,\Phi,{\bf q}){\bf q}\right)=0$  as 
\begin{equation}\label{lin-1}
\begin{split}
\sum_{i,j=1}^2a_{ij}\partial_{ij}\varphi+\sum_{i=1}^2 a_i\partial_i\varphi
=&-\frac{1}{S}\left(\Phi-\frac{1}{2}|{\bf q}|^2\right)\frac{DS\cdot{\bf q}}{q_1^2-c^2(\Phi,{\bf q})}+\frac{D\Phi\cdot{\bf q}}{q_1^2-c^2(\Phi,{\bf q})}\\
&-\frac{{\bf q}^T\cdot{\bf q}\cdot D{\bf t}}{q_1^2-c^2(\Phi,{\bf q})}+\frac{c^2(\mbox{div}_{\bf x}{\bf t})}{q_1^2-c^2(\Phi,{\bf q})}
\end{split}
\end{equation}
for $a_{ij}=a_{ij}(\Phi,{\bf q})$ and $a_i=a_i(\Phi,{\bf q})$ $(i,j=1,2)$ defined by
\begin{equation*}
\begin{split}
&a_{11}:=1,\quad a_{12}(\Phi,{\bf q}):=\frac{1}{r}\frac{q_1q_2}{q_1^2-c^2(\Phi,{\bf q})},\quad a_{21}(\Phi,{\bf q}):=\frac{1}{r}\frac{q_2q_1}{q_1^2-c^2(\Phi,{\bf q})},\\
& a_{22}(\Phi,{\bf q}):=\frac{1}{r^2}\frac{q_2^2-c^2(\Phi,{\bf q})}{q_1^2-c^2(\Phi,{\bf q})},\\
&a_1(\Phi,{\bf q}):=-\frac{2c^2(\Phi,{\bf q})}{r}\frac{1}{q_1^2-c^2(\Phi,{\bf q})},\\
& a_2(\Phi,{\bf q}):=-\frac{1}{r^2}\frac{q_1q_2}{q_1^2-c^2(\Phi,{\bf q})}-\frac{c^2(\Phi,{\bf q})\cos\phi}{r^2\sin\phi}\frac{1}{q_1^2-c^2(\Phi,{\bf q})}.
\end{split}
\end{equation*}
The background solution $(\bar{\varphi},\bar{\Phi})$ satisfies 
\begin{equation}\label{back-1}
\partial_{rr}\bar{\varphi}+\bar{a}_1\partial_r\bar{\varphi}
=\frac{(\partial_r\bar{\varphi})(\partial_r\bar{\Phi})}{|\partial_r\bar{\varphi}|^2-c^2(\bar{\Phi},\partial_r\bar{\varphi})}
\end{equation}
for $\bar{a}_1$ defined by 
\begin{equation}\label{aij-back}
\bar{a}_1:=a_1(\bar{\Phi},\nabla\bar{\varphi}).
\end{equation}
Then,
subtracting \eqref{back-1} from \eqref{lin-1} gives 
\begin{equation}\label{lin-2}
\begin{split}
\sum_{i,j=1}^2a_{ij}\partial_{ij}\chi+\sum_{i=1}^2 a_i\partial_i\chi
=&\,\mathfrak{F}_1(S,DS, \Psi+\bar{\Phi},D\chi+D\bar{\varphi}+{\bf t})\\
&+\mathfrak{F}_2(\Psi+\bar{\Phi},D(\Psi+\bar{\Phi}),D\chi+D\bar{\varphi}+{\bf t})\\
&-\mathfrak{F}_2(\bar{\Phi},{\bf 0},D\bar{\varphi})
\end{split}
\end{equation}
with $\mathfrak{F}_k$ $(k=1,2)$ defined by 
\begin{equation*}
\begin{split}
&\mathfrak{F}_1(\eta,{\bf w},z,{\bf q}):=-\frac{1}{\eta}\left(z-\frac{1}{2}|{\bf q}|^2\right)\frac{{\bf w}\cdot{\bf q}}{q_1^2-c^2(z,{\bf q})}
-\frac{{\bf q}^T\cdot{\bf q}\cdot D{\bf t}}{q_1^2-c^2(z,{\bf q})}+\frac{c^2(\mbox{div}_{\bf x}{\bf t})}{q_1^2-c^2(z,{\bf q})},\\
&\mathfrak{F}_2(z,{\bf p},{\bf q}):=\frac{{\bf p}\cdot{\bf q}}{q_1^2-c^2(z,{\bf q})}-a_1(z,{\bf q})\partial_r\bar{\varphi}
\end{split}
\end{equation*}
for $\eta, z\in\mathbb{R}$, ${\bf w}, {\bf p},{\bf q}\in\mathbb{R}^3$.
For ${\bf V}_0:=(\bar{\Phi},D\bar{\Phi},D\bar{\varphi})$,
we set 
\begin{equation}\label{alpha}
\begin{split}
\alpha_1:=\partial_{q_1}\mathfrak{F}_2({\bf V}_0)
=&\frac{-(\partial_r\bar{\varphi})^2\left(\gamma(\partial_r\bar{\Phi})+\frac{2(\gamma-1)}{r}(\partial_r\bar{\varphi})^2\right)}{(|\partial_r\bar{\varphi}|^2-c^2(\bar{\Phi},\partial_r\bar{\varphi}))^2}\\
&-\frac{c^2(\bar{\Phi},\partial_r\bar{\varphi})\left(\partial_r\bar{\Phi}-\frac{4}{r}(\partial_r\bar{\varphi})^2\right)}{(|\partial_r\bar{\varphi}|^2-c^2(\bar{\Phi},\partial_r\bar{\varphi}))^2},\quad\alpha_2:=0,\\
b_1:=\partial_{p_1}\mathfrak{F}_2({\bf V}_0)=&\frac{\partial_r\bar{\varphi}}{|\partial_r\bar{\varphi}|^2-c^2(\bar{\Phi},\partial_r\bar{\varphi})},\quad
b_2:=0,\\
c:=\partial_z\mathfrak{F}_2({\bf V}_0)=&\frac{(\gamma-1)\nabla\bar{\varphi}\cdot\nabla\bar{\Phi}}{(|\partial_r\bar{\varphi}|^2-c^2(\bar{\Phi},\partial_r\bar{\varphi}))^2}+\frac{2(\gamma-1)(\partial_r\bar{\varphi})^3}{r(|\partial_r\bar{\varphi}|^2-c^2(\bar{\Phi},\partial_r\bar{\varphi}))^2}
\end{split}
\end{equation}
to rewrite \eqref{lin-2} as
\begin{equation}\label{lin-3}
\begin{split}
\mathscr{L}_1(\chi,\Psi)=\mathcal{F}(S,DS, \Psi,D\Psi,D\chi,{\bf t})
\end{split}
\end{equation}
for $\mathscr{L}_1$ and $\mathcal{F}$ defined by 
\begin{equation*}
\begin{split}
&\mathscr{L}_1(\chi,\Psi):=\sum_{i,j=1}^2a_{ij}\partial_{ij}\chi+\sum_{i=1}^2 ((a_i-\alpha_i)\partial_i\chi-b_i\partial_i\Psi)-c\Psi,\\
&
	\begin{split}
	\mathcal{F}(S,DS, \Psi,D\Psi,D\chi,{\bf t})
	&:=\mathfrak{F}_1(S,DS, \Psi+\bar{\Phi},D\chi+D\bar{\varphi}+{\bf t})\\
	&\quad+\mathfrak{F}_2(\Psi+\bar{\Phi},D(\Psi+\bar{\Phi}),D\chi+D\bar{\varphi}+{\bf t})\\
	&\quad-\mathfrak{F}_2(\bar{\Phi},{\bf 0},D\bar{\varphi})-\sum_{i=1}^2 (\alpha_i\partial_i\chi+b_i\partial_i\Psi)-c\Psi.
	\end{split}
\end{split}
\end{equation*}

Next, we rewrite the equation $\Delta\Phi=\varrho(S,\Phi,{\bf q})-b$ as 
\begin{equation}\label{lin-4}
\mathscr{L}_2(\chi,\Psi)=\mathfrak{f}(S,\Psi,D\chi,{\bf t})
\end{equation}
for $\mathscr{L}_2$ and $\mathfrak{f}$ defined by 
\begin{equation}\label{l2l2}
\begin{split}
&\mathscr{L}_2(\chi,\Psi):=\sum_{i,j=1}^2g_{ij}\partial_{ij}\Psi+\sum_{i=1}^2g_i\partial_i\Psi
-g_0\Psi-h_1\partial_r\chi,\\
&\mathfrak{f}(S,\Psi,D\chi,{\bf t})
:=B(S,\Psi+\bar{\Phi},D\chi+D\bar{\varphi},{\bf t})-B(S_0,\bar{\Phi},D\bar{\varphi},{\bf0})\\
&\qquad\qquad\quad\qquad-(b-b_0)
-g_0\Psi-h_1\partial_r\chi,
\end{split}
\end{equation}
where $g_{ij}$ $(i,j=1,2$), $g_i$ ($i=1,2$), $B$, $g_0$, and $h_1$ are defined by 
\begin{equation}\label{gij}
\begin{split}
&\left[\begin{array}{cc}
	g_{11}&g_{12}\\
	g_{21}&g_{22}\end{array}\right] :=\left[\begin{array}{cc}
	1&0\\
	0&\frac{1}{r^2}\end{array}\right],\quad g_1:=\frac{2}{r},\quad g_2:=\frac{\cos\phi}	{r^2\sin\phi},\\
&B(\eta,z,{\bf s},{\bf t}):=\varrho(\eta,z,{\bf s}+{\bf t})\mbox{ for $\eta, z\in\mathbb{R}$, ${\bf s}, {\bf t}\in\mathbb{R}^3$},\\
&g_0:=\partial_zB(S_0,\bar{\Phi},D\bar{\varphi},{\bf0})=\frac{1}{\gamma S_0}\varrho^{2-\gamma}(S_0,\bar{\Phi},D\bar{\varphi}),\\
&h_1:=\partial_{s_1}B(S_0,\bar{\Phi},D\bar{\varphi},{\bf0})=-\frac{1}{\gamma S_0}\varrho^{2-\gamma}(S_0,\bar{\Phi},D\bar{\varphi})\partial_r\bar{\varphi}
\end{split}
\end{equation}
for $\varrho$ given in \eqref{def-H-G}.

Then, by \eqref{lin-3} and \eqref{lin-4}, the system \eqref{C-N-lin} can be rewritten as 
\begin{equation}\label{lin-pps}
\left\{\begin{split}
&\mathscr{L}_1(\chi,\Psi)=\mathcal{F}(S,DS, \Psi,D\Psi,D\chi,{\bf t}),\\
&\mathscr{L}_2(\chi,\Psi)=\mathfrak{f}(S,\Psi,D\chi,{\bf t}).
\end{split}\right.
\end{equation}
Also, one can directly check that the boundary conditions in \eqref{pp-prob-sup} are rewritten as 
\begin{equation}\label{lin-pps-bd}
\left\{\begin{split}
\chi=\varphi_{\rm en}-\bar{\varphi},\,\,\nabla\chi\cdot{\bf e}_r=-\partial_r\bar{\varphi}(r_{\rm en})+u_{\rm en}-[\nabla\times(\psi_{\ast}{\bf e}_{\theta})]\cdot{\bf e}_r\,\, &\mbox{on}\,\,\Gamma_{\rm en},\\
\nabla\Psi\cdot{\bf e}_r=E_{\rm en}-E_0\,\, &\mbox{on}\,\,\Gamma_{\rm en},\\
\nabla\chi\cdot{\bf e}_{\phi}=0,\,\,
\nabla\Psi\cdot{\bf e}_{\phi}=0\,\,&\mbox{on}\,\,\Gamma_{\rm w},\\
\nabla\Psi\cdot{\bf e}_r=\Phi_{\rm ex}-\bar{\Phi}\,\,&\mbox{on}\,\,\Gamma_{\rm ex}.
\end{split}\right.
\end{equation}
%
For a fixed $(\chi_{\ast},\Psi_{\ast})\in\mathcal{J}(\delta_2)$, define $\mathscr{L}_1^{(\chi_{\ast},\Psi_{\ast})}$ and ${\bf t}^{\ast}$ by
\begin{equation*}\label{fix-lin-pro}
\begin{split}
&\mathscr{L}_1^{(\chi_{\ast},\Psi_{\ast})}(\chi,\Psi):=\sum_{i,j=1}^2\tilde{a}_{ij}\partial_{ij}\chi+\sum_{i=1}^2 ((\tilde{a}_i-\alpha_i)\partial_i\chi-b_i\partial_i\Psi)-c\Psi,\\
&{\bf t}^{\ast}:={\bf t}\left(\nabla\times(\psi_{\ast}{\bf e}_{\theta}),\frac{\Lambda_{\ast}}{r\sin\phi}\right)
\end{split}
\end{equation*}
with
\begin{equation*}
\begin{split}
&\tilde{a}_{ij}:=a_{ij}(\Psi_{\ast}+\bar{\Phi},\tilde{\bf q}),\quad \tilde{a}_{i}:=a_{i}(\Psi_{\ast}+\bar{\Phi},\tilde{\bf q}),\\
&\tilde{\bf q}:={\bf q}\left(\nabla(\chi_{\ast}+\bar{\varphi}),\nabla\times(\psi_{\ast}{\bf e}_{\theta}),\frac{\Lambda_{\ast}}{r\sin\phi}\right).
\end{split}
\end{equation*}
We first solve the following problem:
\begin{equation}\label{fix-lin-pps}
\left\{\begin{split}
&\mathscr{L}_1^{(\chi_{\ast},\Psi_{\ast})}(\chi,\Psi)=\mathcal{F}(S_{\ast},DS_{\ast}, \Psi_{\ast},D\Psi_{\ast},D\chi_{\ast},{\bf t}^{\ast})\\
&\mathscr{L}_2(\chi,\Psi)=\mathfrak{f}(S_{\ast},\Psi_{\ast},D\chi_{\ast},{\bf t}^{\ast})
\end{split}\right.\quad\mbox{in }\Omega
\end{equation}
with the boundary conditions \eqref{lin-pps-bd}.

To simplify the boundary conditions of $\chi$ and $\Psi$ in \eqref{lin-pps-bd}, define functions $\chi_{bd}$, $\chi_{h}$, $\Psi_{bd}$, and $\Psi_{h}$ by 
\begin{equation*}
\begin{split}
&\chi_{bd}(r,\phi):=\eta(r)\left(\varphi_{\rm en}(\phi)-\bar{\varphi}(r_{\rm en})\right)=\eta(r)\left(\varphi_{\rm en}(\phi)-\frac{m_0^2}{r_{\rm en}^2\rho_0}\right),\\
&\chi_{h}(r,\phi):=\chi(r,\phi)-\chi_{bd}(r,\phi),\\
&\Psi_{bd}(r,\phi):=r\left[\frac{r_{\rm ex}(E_{\rm en}(\phi)-E_0)-r_{\rm en}(\Phi_{\rm ex}(\phi)-\bar{\Phi}(r_{\rm ex},\phi))}{r_{\rm ex}-r_{\rm en}}\right]\\
	&\qquad\qquad\quad+\frac{r^2}{2}\left[\frac{(\Phi_{\rm ex}(\phi)-\bar{\Phi}(r_{\rm ex},\phi))-(E_{\rm en}(\phi)-E_0)}{r_{\rm ex}-r_{\rm en}}\right],\\
&\Psi_h(r,\phi):=\Psi(r,\phi)-\Psi_{bd}(r,\phi)
\end{split}
\end{equation*}
for a smooth function $\eta=\eta(r)$ satisfying the following properties:
\begin{equation*}
\begin{split}
&0\le \eta(r)\le 1\mbox{ for }r\in[r_{\rm en},r_{\rm ex}],\\
&\eta(r)=\left\{\begin{split}
1\quad&\mbox{on }r=r_{\rm en},\\
0\quad&\mbox{on }r=r_{\rm ex},
\end{split}\right.\quad \mbox{ and }\left|\frac{d^k\eta}{dr^k}\right|\le 10\left(1+\frac{1}{L}\right),\,k=1,2,3,4.
\end{split}
\end{equation*}
One can easily check that if $(\chi_h,\Psi_h)$ satisfies 
\begin{equation}\label{fix-lin-ho}
\left\{\begin{split}
\mathscr{L}_1^{(\chi_{\ast},\Psi_{\ast})}(\chi_h,\Psi_h)=\mathcal{F}_h,\quad \mathscr{L}_2(\chi_h,\Psi_h)=\mathfrak{f}_h\quad&\mbox{in}\,\,\Omega,\\
\chi_h=0,\quad\nabla\chi_h\cdot{\bf e}_r=\mathfrak{g}_h,\quad
\nabla\Psi_h\cdot{\bf e}_r=0\quad &\mbox{on}\,\,\Gamma_{\rm en},\\
\nabla\chi_h\cdot{\bf e}_\phi=0,\quad
\nabla\Psi_h\cdot{\bf e}_{\phi}=0\quad&\mbox{on}\,\,\Gamma_{\rm w},\\
\nabla\Psi_h\cdot{\bf e}_r=0\quad&\mbox{on}\,\,\Gamma_{\rm ex}
\end{split}\right.
\end{equation}
for $\mathcal{F}_h$, $\mathfrak{f}_h$, and $\mathfrak{g}_h$ defined by 
\begin{equation*}
\begin{split}
\mathcal{F}_h&:=\mathcal{F}(S_{\ast},DS_{\ast}, \Psi_{\ast},D\Psi_{\ast},D\chi_{\ast},{\bf t}^{\ast})-\mathscr{L}_1^{(\chi_{\ast},\Psi_{\ast})}(\chi_{bd},\Psi_{bd}),\\
\mathfrak{f}_h&:=\mathfrak{f}(S_{\ast},\Psi_{\ast},D\chi_{\ast},{\bf t}^{\ast})-\mathscr{L}_2(\chi_{bd},\Psi_{bd}),\\
\mathfrak{g}_h&:=-\partial_r\bar{\varphi} (r_{\rm en})+u_{\rm en}-[\nabla\times(\psi_{\ast}{\bf e}_{\theta})]\cdot{\bf e}_r-\nabla\chi_{bd}\cdot{\bf e}_r,
\end{split}
\end{equation*}
then $(\chi,\Psi)$ satisfies \eqref{lin-pps-bd}-\eqref{fix-lin-pps}. 
Before we solve \eqref{fix-lin-ho}, we point out some properties in the following lemma.
\begin{lemma}\label{Lem51}
For a constant $L(=r_{\rm ex}-r_{\rm en})\in(0,r_{\ast}-r_{\rm en}]$ to be determined, the following properties hold:
\begin{itemize}
\item[(a)]  $\bar{a}_1$, $\alpha_1$, $b_1$, $c$, $g_{11}$, $g_{22}$, $g_1$, $g_0$, and $h_1$ given in \eqref{aij-back}, \eqref{alpha}, and \eqref{gij} are smooth in $\overline{\Omega}$, and for each $l\in\mathbb{N}$, there exists a constant $M_l>0$ depending only on the data and $l$ to satisfy
\begin{equation*}
\|(\bar{a}_1,\alpha_1,b_1,c,g_{11}, g_{22}, g_1,g_0,h_1)\|_{C^l(\overline{\Omega})}\le M_l.
\end{equation*}

\item[(b)] There exists a constant $\nu\in(0,1)$ depending only on the data so that the inequality
\begin{equation*}
\nu\le -\bar{a}_{22}=\frac{1}{r^2}\frac{c^2(\bar{\Phi},\nabla\bar{\varphi})}{(\partial_r\bar{\varphi})^2-c^2(\bar{\Phi},\nabla\bar{\varphi})}\le \frac{1}{\nu}
\end{equation*}
holds in $\overline{\Omega}$.

\item[(c)] There exists a small constant ${\bm \tau}_{\flat}>0$ depending only on the data and $(\delta_1,\delta_2,\delta_3)$ so that if 
$${\bm\tau}_p\le {\bm\tau}_{\flat},$$
then the following estimates hold:
$$\frac{\nu}{2}\le -\bar{a}_{22}=\frac{1}{r^2}\frac{c^2(\bar{\Phi},\nabla\bar{\varphi})}{(\partial_r\bar{\varphi})^2-c^2(\bar{\Phi},\nabla\bar{\varphi})}\le \frac{2}{\nu},$$
and
\begin{equation}\label{est-coeff}
\begin{split}
&\|\tilde{a}_{ij}-\bar{a}_{ij}\|_{H^3_{\ast}(\Omega)}\le C\left(\|\chi_{\ast}\|_{H^4_{\ast}(\Omega)}+\|\Psi_{\ast}\|_{H^4_{\ast}(\Omega)}+\mathfrak{V}_2\right),\\
&\|\mathcal{F}_h\|_{H^3_{\ast}(\Omega)}\le  C\left({\bm\tau}_p+\|\chi_{\ast}\|^2_{H^4_{\ast}(\Omega)}+\|\Psi_{\ast}\|^2_{H^4_{\ast}(\Omega)}+\mathfrak{V}_3\right),\\
&\|\mathfrak{f}_h\|_{H^2(\Omega)}\le C\left({\bm\tau}_p+\|\chi_{\ast}\|^2_{H^4_{\ast}(\Omega)}+\|\Psi_{\ast}\|^2_{H^4_{\ast}(\Omega)}+\mathfrak{V}_3\right),\\
&\|\mathfrak{g}_h\|_{H^3(\Gamma_{\rm en})}\le C\left({\bm\tau}_p+\mathfrak{V}_1\right)
\end{split}
\end{equation}
for a constant $C>0$ depending only on the data. 
In \eqref{est-coeff}, $\mathfrak{V}_1$, $\mathfrak{V}_2$, and $\mathfrak{V}_3$ are given by 
\begin{equation*}
\begin{split}
&\mathfrak{V}_1:=\|\psi_{\ast}{\bf e}_{\theta}\|_{H^5_{\ast}(\Omega;\mathbb{R}^3)},\\
&\mathfrak{V}_2:=\|\psi_{\ast}{\bf e}_{\theta}\|_{H^5_{\ast}(\Omega;\mathbb{R}^3)}+\left\|\frac{\Lambda_{\ast}}{r\sin\phi}{\bf e}_{\theta}\right\|_{H^3_{\ast}(\Omega;\mathbb{R}^3)},\\
&\mathfrak{V}_3:=\|\psi_{\ast}{\bf e}_{\theta}\|_{H^5_{\ast}(\Omega;\mathbb{R}^3)}+\left\|\frac{\Lambda_{\ast}}{r\sin\phi}{\bf e}_{\theta}\right\|_{H^3_{\ast}(\Omega;\mathbb{R}^3)}+\|S_{\ast}-S_0\|_{H^4_{\ast}(\Omega)}.
\end{split}
\end{equation*}
\end{itemize}
\end{lemma}

\begin{lemma}[A priori $H^1$-estimate]\label{lem-H1}
For ${\bm\tau}_{\flat}$ in Lemma \ref{Lem51},
suppose that 
$${\bm\tau}_p\in(0, {\bm\tau}_{\flat}].$$
For $r_{\ast}>0$ in Lemma \ref{super-lemma}, 
there exists a constant $L(=r_{\rm ex}-r_{\rm en})\in(0,r_{\ast}-r_{\rm en}]$ depending only on the data so that 
if $(\chi_h,\Psi_h)\in [C^2(\Omega)]^2$ solves the linear boundary value problem \eqref{fix-lin-ho} associated with $(\chi_{\ast},\Psi_{\ast})\in\mathcal{J}(\delta_2)$, then it satisfies the estimate 
\begin{equation*}
\|\chi_h\|_{H^1(\Omega)}+\|\Psi_h\|_{H^1(\Omega)}\le C\left(\|\mathcal{F}_h\|_{L^2(\Omega)}+\|\mathfrak{f}_h\|_{L^2(\Omega)}+\|\mathfrak{g}_h\|_{L^2({\Gamma_{\rm en}})}\right)
\end{equation*}
for a constant $C>0$ depending only on the data and $L$.
\end{lemma}
\begin{proof} 

For a smooth positive function $\mathcal{W}=\mathcal{W}(r)$ to be determined later, set 
\begin{equation*}
\mathcal{I}_L(\chi_h,\Psi_h,\mathcal{W}):=\int_{\mathcal{N}_L}\left[\mathcal{W}(\partial_r\chi_h)\mathscr{L}_1^{(\chi_{\ast},\Psi_{\ast})}(\chi_h,\Psi_h)-\Psi_h\mathscr{L}_2(\chi_h,\Psi_h)\right]r^2\sin\phi d{\bf y},
\end{equation*}
where $\mathcal{N}_L$ is a three-dimensional rectangular domain defined by 
\begin{equation*}
\begin{split}
\mathcal{N}_L:=\{{\bf y}=(r,\phi,\theta)\in\mathbb{R}^3:\, r_{\rm en}<r<r_{\rm ex},\,0<\phi <\phi_0,\, 0<\theta<2\pi\}.
\end{split}
\end{equation*}
Integrating by parts gives 
\begin{equation*}
\mathcal{I}_L(\chi_h,\Psi_h,\mathcal{W})=J_1+J_2+J_3
\end{equation*}
with
\begin{equation*}
\begin{split}
J_1:=&\int_{\mathcal{N}_L}\left[\mathfrak{a}_1\frac{(\partial_r\chi_h)^2}{2}+\mathfrak{a}_2\frac{(\partial_{\phi}\chi_h)^2}{2}+(\partial_r\Psi_h)^2+\frac{(\partial_\phi\Psi_h)^2}{r^2}+g_0\Psi_h^2\right](r^2\sin\phi)d{\bf y},\\
J_2:=&\int_{\mathcal{N}_L}\mathcal{W}\mathfrak{a}_3(\partial_r\chi_h)(\partial_\phi\chi_h)(r^2\sin\phi)d{\bf y}
-\int_{\mathcal{N}_L}\mathcal{W}b_1(\partial_r\chi_h)(\partial_r\Psi_h)(r^2\sin\phi)d{\bf y}\\
&+\int_{\mathcal{N}_L}\left(-c\mathcal{W}+h_1\right)(\partial_r\chi_h)\Psi_h(r^2\sin\phi)d{\bf y},\\
J_3:=&\int_{\partial\mathcal{N}_L\cap\{r=r_{\rm ex}\}}\mathcal{W}\left[\frac{(\partial_r\chi_h)^2}{2}-\tilde{a}_{22}\frac{(\partial_{\phi}\chi_h)^2}{2}\right](r^2\sin\phi )d{\bf y}'\\
&-\int_{\partial\mathcal{N}_L\cap\{r=r_{\rm en}\}}\mathcal{W}\frac{\mathfrak{g}_h^2}{2}(r^2\sin\phi )d{\bf y}',
\end{split}
\end{equation*}
where $\mathfrak{a}_1$, $\mathfrak{a}_2$, and $\mathfrak{a}_3$ are defined by 
\begin{equation}\label{a123}
\begin{split}
\mathfrak{a}_1:=&\,2\mathcal{W}(\tilde{a}_1-\alpha_1)-\left[\partial_r\mathcal{W}+\frac{2\mathcal{W}}{r}+\frac{2\mathcal{W}\tilde{a}_{12}\cos\phi}{\sin\phi}+2\mathcal{W}(\partial_{\phi}\tilde{a}_{12})\right],\\
\mathfrak{a}_2:=&\,\partial_r(\mathcal{W}\tilde{a}_{22})+\frac{2\mathcal{W}\tilde{a}_{22}}{r}=-\tilde{a}_{22}\left(-\partial_r\mathcal{W}-\frac{\partial_r\tilde{a}_{22}}{\tilde{a}_{22}}\mathcal{W}-\frac{2\mathcal{W}}{r}\right),\\
\mathfrak{a}_3:=&\,-\left(\partial_{\phi}\tilde{a}_{22}+\frac{\tilde{a}_{22}\cos\phi}{\sin\phi}\right)+\tilde{a}_2\\
	=&\,-\partial_{\phi}\tilde{a}_{22}-\left(\frac{\tilde{\bf q}\cdot{\bf e}_{\phi}}{r^2}\right)\frac{\cot\phi+(\tilde{\bf q}\cdot{\bf e}_r)(\tilde{\bf q}\cdot{\bf e}_\phi)}{\tilde{\bf q}\cdot{\bf e}_r-c^2(\Psi_{\ast}+\bar{\Phi},\tilde{\bf q})}.
\end{split}
\end{equation} 
By the Cauchy-Schwartz inequality, we get 
\begin{equation*}
\begin{split}
|J_2|
\le &\int_{\mathcal{N}_L} \left[\frac{1}{2}\mathcal{W}^2(\partial_r\chi_h)^2+\frac{1}{2}\mathfrak{a}_3^2(\partial_{\phi}\chi_h)^2\right](r^2\sin\phi)d{\bf y}\\
&+\int_{\mathcal{N}_L}\left[2\mathcal{W}^2b_1^2(\partial_r\chi_h)^2+\frac{1}{8}(\partial_r\Psi_h)^2\right](r^2\sin\phi)d{\bf y}\\
&+\int_{\mathcal{N}_L}\left[\frac{1}{\mu}(-c\mathcal{W}+h_1)^2(\partial_r\chi_h)^2+\frac{\mu}{4}\Psi_h^2\right](r^2\sin\phi)d{\bf y}
\end{split}
\end{equation*}
for $\mu:=\inf_{r\in[r_{\rm en},r_{\rm en}+L]} g_0(r)$.
Note that 
\begin{equation*}
\begin{split}
0<\mathfrak{a}_3^2=&\,\left[-\partial_{\phi}\tilde{a}_{22}-\left(\frac{\tilde{\bf q}\cdot{\bf e}_{\phi}}{r^2}\right)\frac{\cot\phi+(\tilde{\bf q}\cdot{\bf e}_r)(\tilde{\bf q}\cdot{\bf e}_\phi)}{\tilde{\bf q}\cdot{\bf e}_r-c^2(\Psi_{\ast}+\bar{\Phi},\tilde{\bf q})}\right]^2
\le 2M_0
\end{split}
\end{equation*}
for a constant $M_0>0$ depending only on the data.
Then
\begin{equation}\label{J2-ab}
\begin{split}
|J_2|\le 
&\int_{\mathcal{N}_L}\left[\frac{1}{8}(\partial_r\Psi_h)^2+\frac{\mu}{4}\Psi_h^2+\mathfrak{a}_4\frac{(\partial_r\chi_h)^2}{2}+
M_0(\partial_{\phi}\chi_h)^2\right](r^2\sin\phi)d{\bf y}\\
\end{split}
\end{equation}
for $\mathfrak{a}_4$  defined by 
\begin{equation}\label{a4}
\begin{split}
&\mathfrak{a}_4:=\frac{1}{2}\mathcal{W}^2+2\mathcal{W}^2b_1^2+\frac{1}{\mu}(-c\mathcal{W}+h_1)^2.
\end{split}
\end{equation}
By using \eqref{J2-ab}, we get 
\begin{equation*}
\begin{split}
\mathcal{I}_L(\chi_h,\Psi_h,\mathcal{W})
\ge& \int_{\mathcal{N}_L} \left(\mathfrak{a}_1-\mathfrak{a}_4\right)\frac{(\partial_r\chi_h)^2}{2}(r^2\sin\phi)d{\bf y}\\
&+\int_{\mathcal{N}_L}\left(\mathfrak{a}_2-{2M_0}\right)\frac{(\partial_{\phi}\chi_h)^2}{2}(r^2\sin\phi)d{\bf y}\\
&+\int_{\mathcal{N}_L}\left[\frac{7}{8}(\partial_r\Psi_h)^2+\frac{(\partial_{\phi}\Psi_h)^2}{r^2}+\frac{3\mu}{4}\Psi_h^2\right](r^2\sin\phi)d{\bf y}\\
&+\inf_{{\bf y}\in\overline{\mathcal{N}_L}}\{1,-\tilde{a}_{22}({\bf y})\}\int_{\partial\mathcal{N}_L\cap\{r=r_{\rm ex}\}}\frac{\mathcal{W}}{2}\left[(\partial_r\chi_h)^2+(\partial_{\phi}\chi_h)^2\right](r^2\sin\phi)d{\bf y}'\\
&-\int_{\partial\mathcal{N}_L\cap\{r=r_{\rm en}\}}\mathcal{W}\frac{\mathfrak{g}_h^2}{2}(r^2\sin\phi )d{\bf y}'.
\end{split}
\end{equation*}
If we choose $\mathcal{W}$  and $L$ satisfying 
\begin{equation}\label{W-sa}
\left\{\begin{split}
&\mathcal{W}(r)>0\mbox{ for } r\in[r_{\rm en},  r_{\rm en}+ L],\\
&\mathfrak{a}_1-\mathfrak{a}_4\ge \lambda_0\mbox{ in }\overline{\mathcal{N}_L},\\
&-\frac{1}{\tilde{a}_{22}}(\mathfrak{a}_2-{2M_0})\ge {\lambda_0}\mbox{ in }\overline{\mathcal{N}_L}
\end{split}\right.
\end{equation}
for some constant $\lambda_0>0$, then we get 
\begin{equation}\label{IL-left}
\begin{split}
\mathcal{I}_L(\chi_h,\Psi_h,\mathcal{W})
\ge&\frac{\lambda_0}{2}\int_{\mathcal{N}_L}\left[(\partial_r\chi_h)^2+\inf_{{\bf y}\in\overline{\mathcal{N}_L}}(-\tilde{a}_{22}({\bf y}))(\partial_{\phi}\chi_h)^2\right](r^2\sin\phi)d{\bf y}\\
&+\frac{7}{8}\int_{\mathcal{N}_L}\left[(\partial_r\Psi_h)^2+\frac{1}{r^2}(\partial_\phi\Psi_h)^2\right](r^2\sin\phi)d{\bf y}\\
&+\frac{3\mu}{4}\int_{\mathcal{N}_L}\Psi_h^2 (r^2\sin\phi)d{\bf y}\\
&-\int_{\partial\mathcal{N}_L\cap\{r=r_{\rm en}\}}\mathcal{W}\frac{\mathfrak{g}_h^2}{2}(r^2\sin\phi )d{\bf y}'.
\end{split}
\end{equation}
Since 
\begin{equation*}
\mathcal{I}_L(\chi_h,\Psi_h,\mathcal{W})=\int_{\mathcal{N}_L}\left[\mathcal{W}(\partial_r\chi_h)\mathcal{F}_h-\Psi_h\mathfrak{f}_h\right](r^2\sin\phi )d{\bf y},
\end{equation*}
the Cauchy-Schwarz inequality implies that 
\begin{equation}\label{IL-right}
\mathcal{I}_L(\chi_h,\Psi_h,\mathcal{W})\le \int_{\mathcal{N}_L}\left[\frac{\lambda_0}{4}(\partial_r\chi_h)^2+\frac{\mu}{4}\Psi_h^2+\frac{1}{\lambda_0}|\mathcal{W}\mathcal{F}_h|^2+\frac{1}{\mu}|\mathfrak{f}_h|^2 \right](r^2\sin\phi) d{\bf y}.
\end{equation}
It follows from \eqref{IL-left}-\eqref{IL-right} that 
\begin{equation*}
\begin{split}
&\frac{\lambda_0}{4}\inf_{{\bf y}\in\overline{\mathcal{N}_L}}\{1,-r^2\tilde{a}_{22}({\bf y})\}\int_{\mathcal{N}_L}\left[(\partial_r\chi_h)^2+\frac{(\partial_{\phi}\chi_h)^2}{r^2}\right](r^2\sin\phi)d{\bf y}\\
&\quad+\frac{7}{8}\int_{\mathcal{N}_L}\left[(\partial_r\Psi_h)^2+\frac{(\partial_{\phi}\Psi_h)^2}{r^2}\right](r^2\sin\phi)d{\bf y}+\frac{\mu}{2}\int_{\mathcal{N}_L}\Psi_h^2 (r^2\sin\phi)d{\bf y}\\
&\le \frac{1}{\lambda_0}\int_{\mathcal{N}_L}|\mathcal{W}\mathcal{F}_h|^2(r^2\sin\phi)d{\bf y}+\frac{1}{\mu}\int_{\mathcal{N}_L}|\mathfrak{f}_h|^2 (r^2\sin\phi)d{\bf y}\\
&\quad+\int_{\partial\mathcal{N}_L\cap\{r=r_{\rm en}\}}\mathcal{W}\frac{\mathfrak{g}_h^2}{2}(r^2\sin\phi )d{\bf y}'.
\end{split}
\end{equation*}
Since $\chi_{h}=0$ on $\partial\mathcal{N}_L\cap\{r=r_{\rm en}\}$, we apply the Poincar\'e inequality to get 
\begin{equation*}
\begin{split}
&\|\chi_h\|_{H^1(\Omega)}+\|\Psi_h\|_{H^1(\Omega)}\\
&\le C\left(\|\mathcal{F}_h\|_{L^2(\Omega)}+\|\mathfrak{f}_h\|_{L^2(\Omega)}+\|\mathfrak{g}_h\|_{L^2({\Gamma_{\rm en}})}+\|\varphi_{\rm en}-\bar{\varphi}\|_{C^0(\overline{\Gamma_{\rm en}})}\right)
\end{split}
\end{equation*}
for a constant $C>0$ depending only on the data, $L$, $\lambda_0$, and $\max_{r\in[r_{\rm en},r_{\rm ex}]}\mathcal{W}(r)$.

To complete the proof, we find $\mathcal{W}$ and $L$ satisfying \eqref{W-sa}.
The definitions of $\mathfrak{a}_i$ $(i=1,2,3,4)$ in \eqref{a123} and \eqref{a4} directly imply
\begin{equation*}
\begin{split}
\mathfrak{a}_1-\mathfrak{a}_4
=&\,2\mathcal{W}(\tilde{a}_1-\alpha_1)-\left[\partial_r\mathcal{W}+\frac{2\mathcal{W}}{r}+\frac{2\mathcal{W}\tilde{a}_{12}\cos\phi}{\sin\phi}+2\mathcal{W}(\partial_{\phi}\tilde{a}_{12})\right]\\
&-\frac{1}{2}\mathcal{W}^2-2\mathcal{W}^2b_1^2-\frac{1}{\mu}(-c\mathcal{W}+h_1)^2\\
=&-\partial_r\mathcal{W}+\left[2(\tilde{a}_1-\alpha_1)-\frac{2}{r}-\frac{2\tilde{a}_{12}\cos\phi}{\sin\phi}-2(\partial_{\phi}\tilde{a}_{12})+\frac{2ch_1}{\mu}\right]\mathcal{W}\\
&-\left[\frac{1}{2}+2b_1^2+\frac{c^2}{\mu}\right]\mathcal{W}^2-\frac{h_1^2}{\mu}
\end{split}
\end{equation*}
and
\begin{equation*}
\begin{split}
-\frac{1}{\tilde{a}_{22}}(\mathfrak{a}_2-{2M_0})
=&-\partial_r\mathcal{W}+\left(-\frac{\partial_r\tilde{a}_{22}}{\tilde{a}_{22}}-\frac{2}{r}\right)\mathcal{W}+\frac{2M_0}{\tilde{a}_{22}}.
\end{split}
\end{equation*}
By applying the method of \cite[Step 3 in the proof of Proposition 2.4]{bae2021structural}, we can find $\mathcal{W}$ and $L$ satisfying 
\begin{equation*}
-\partial_r\mathcal{W}-\mathfrak{b}_2\mathcal{W}^2-\mathfrak{b}_1\mathcal{W}-\mathfrak{b}_0=\lambda_0\mbox{ for }r\in[r_{\rm en},r_{\rm ex}(=r_{\rm en}+L)],
\end{equation*}
where $\mathfrak{b}_2$, $\mathfrak{b}_1$, and $\mathfrak{b}_0$ are defined by  
\begin{equation*}
\begin{split}
&\mathfrak{b}_2:=\max_{r\in[r_{\rm en},r_{\rm ex}]}\left(\frac{1}{2}+2b_1^2+\frac{c^2}{\mu}\right),\\
&\mathfrak{b}_1:=\max_{{\bf y}\in\overline{\mathcal{N}_L}}\left\{-\left[2(\tilde{a}_1-\alpha_1)-\frac{2}{r}-\frac{2\tilde{a}_{12}\cos\phi}{\sin\phi}-2(\partial_{\phi}\tilde{a}_{12})+\frac{2ch_1}{\mu}\right],\frac{\partial_r\tilde{a}_{22}}{\tilde{a}_{22}}+\frac{2}{r}\right\},\\
&\mathfrak{b}_0:=\max_{{\bf y}\in\overline{\mathcal{N}_L}}\left\{\frac{h_1^2}{\mu}, \frac{2M_0}{-\tilde{a}_{22}}\right\}.
\end{split}
\end{equation*}
Furthermore, $\lambda_0$ and $\max_{r\in[r_{\rm en},r_{\rm ex}]}\mathcal{W}(r)$ depend on the data and $L$.
The proof of Lemma \ref{lem-H1} is completed.
\end{proof}

\begin{lemma}\label{Lem-po}
Let ${\bm\tau}_{\flat}$ be from Lemma \ref{Lem51}.
There exists a small constant $\tilde{\bm\tau}_{\flat}\in(0,{\bm\tau_{\flat}}]$ depending only on the data and $(L,\delta_1,\delta_2,\delta_3)$ so that if ${\bm\tau}_p\in(0,\tilde{\bm\tau}_{\flat}]$, then, 
for each $(\chi_{\ast},\Psi_{\ast})\in\mathcal{J}(\delta_2)$,
the linear boundary value problem \eqref{fix-lin-ho} has a unique axisymmetric solution $(\chi_h,\Psi_h)\in[H_{\ast}^4(\Omega)]^2$ that satisfies 
\begin{equation}\label{lem-vw-est}
\begin{split}
&\|\chi_h\|_{H_{\ast}^4(\Omega)}\le C\left(\|\mathcal{F}_h\|_{H_{\ast}^3(\Omega)}+\|\mathfrak{f}_h\|_{H^2(\Omega)}+\|\mathfrak{g}_h\|_{H^3({\Gamma_{\rm en}})}\right),\\
&\|\Psi_h\|_{H_{\ast}^4(\Omega)}\le C\left(\|\mathcal{F}_h\|_{H^2(\Omega)}+\|\mathfrak{f}_h\|_{H^2(\Omega)}+\|\mathfrak{g}_h\|_{H^2({\Gamma_{\rm en}})}\right)
\end{split}
\end{equation}
for a constant $C>0$ depending only on the data and $L$.
\end{lemma}
\begin{proof}
Overall, the outline is same as the proof of \cite[Proposition 2.6]{bae2021structural}. 
In this case, however, we need to compute estimates very carefully due to the singularity issues related with the polar angle $\phi$ on the axis of the divergent nozzle.

{\bf Step 1.} Approximation of \eqref{fix-lin-ho}: 
By using standard mollifier and the extension method (cf. \cite[$\S$6.9]{gilbarg2015elliptic}), one can show that,  
for each $(S_{\ast},\frac{\Lambda_{\ast}}{r\sin\phi}{\bf e}_{\theta},\chi_{\ast},\Psi_{\ast},{\bf W}_{\ast})\in\mathcal{J}(\delta_1)\times\mathcal{J}(\delta_2)\times\mathcal{J}(\delta_3)$, there exists a sequence $\{(S_n,\frac{\Lambda_n}{r\sin\phi}{\bf e}_{\theta},\chi_n,\Psi_n,{\bf W}_n)\}_{n=1}^{\infty}\subset C^{\infty}(\overline{\Omega})$ such that 
\begin{equation*}
\begin{split}
&\partial_\phi S_n=\partial_{\phi}\left(\frac{\Lambda_n}{r\sin\phi}\right)=\partial_\phi \chi_n=\partial_\phi\Psi_n={\bf W}_n=0\mbox{ on }\Gamma_{\rm w},\\
&\|S_n-S_{\ast}\|_{H_{\ast}^4(\Omega)}+\left\|\left(\frac{\Lambda_n}{r\sin\phi}-\frac{\Lambda_{\ast}}{r\sin\phi}\right){\bf e}_{\theta}\right\|_{H_{\ast}^3(\Omega)}\le \frac{\delta_1{\bm\tau}_p}{2n},\\
&\|\chi_n-\chi_{\ast}\|_{H_{\ast}^4(\Omega)}+\|\Psi_n-\Psi_{\ast}\|_{H_{\ast}^4(\Omega)}\le \frac{\delta_2{\bm\tau}_p}{2n},\\
&\|{\bf W}_n-{\bf W}_{\ast}\|_{H^5_{\ast}(\Omega)}\le \frac{\delta_3{\bm\tau}_p}{2n}.
\end{split}
\end{equation*}
To simplify notations, for each $n\in\mathbb{N}$, we set
\begin{equation*}
\begin{split}
&{\bf q}_n:={\bf q}\left(\nabla(\chi_n+\bar{\varphi}),\nabla\times {\bf W}_n,\frac{\Lambda_n}{r\sin\phi}\right),\quad{\bf t}_n:={\bf q}_n-\nabla(\chi_n+\bar{\varphi}),\\
&a_{ij}^{(n)}:=a_{ij}(\Psi_n+\bar{\Phi},{\bf q}_n),\quad
a_i^{(n)}:=a_i(\Psi_n+\bar{\Phi},{\bf q}_n).
\end{split}
\end{equation*}
Also, we set 
\begin{equation*}
\begin{split}
&\mathcal{L}_1^{(n)}(\chi,\Psi):=\sum_{i,j=1}^2 a_{ij}^{(n)}\partial_{ij}\chi+\sum_{i=1}^2((a_i^{(n)}-\alpha_i)\partial_i\chi-b_i\partial_i\Psi)-c\Psi,\\
&\mathcal{F}^{(n)}:=\mathcal{F}(S_n,DS_n, \Psi_n,D\Psi_n,D\chi_n,{\bf t}_n)-\mathcal{L}_1(\chi_{bd},\Psi_{bd}),\\
&\mathfrak{f}^{(n)}:=\mathfrak{f}(S_n,\Psi_n,D\chi_n,{\bf t}_n)-\mathscr{L}_2(\chi_{bd},\Psi_{bd}),\\
&\mathfrak{g}^{(n)}:=-u_0+u_{\rm en}-(\nabla\times{\bf W}_n)\cdot{\bf e}_r-\nabla\chi_{bd}\cdot{\bf e}_r.
\end{split}
\end{equation*}

\begin{lemma}\label{Lemma27}
There exists a small constant ${\bm\tau}^{\flat}\in(0,{\bm\tau_{\flat}}]$ depending only on the data and $(L,\delta_1,\delta_2,\delta_3)$ so that if ${\bm\tau}_p\in(0,{\bm\tau}^{\flat}]$, then, 
for each $n\in\mathbb{N}$, the linear boundary value problem 
\begin{equation}\label{app-vw}
\left\{\begin{split}
\mathcal{L}_1^{(n)}(v,w)=\mathcal{F}^{(n)},\quad \mathscr{L}_2(v,w)=\mathfrak{f}^{(n)}&\mbox{ in }\Omega,\\
v=0,\quad \partial_r v=\mathfrak{g}^{(n)},\quad \partial_rw=0&\mbox{ on }\Gamma_{\rm en},\\
\partial_\phi v=0,\quad \partial_\phi w=0&\mbox{ on }\Gamma_{\rm w},\\
\partial_rw=0&\mbox{ on }\Gamma_{\rm ex}
\end{split}\right.
\end{equation}
has a unique axisymmetric solution $(v,w)\in[H_{\ast}^4(\Omega)]^2$ that satisfies 
\begin{equation}\label{vw0est}
\begin{split}
&\|v\|_{H^4_{\ast}(\Omega)}\le C\left(\|\mathcal{F}^{(n)}\|_{H^3_{\ast}(\Omega)}+\|\mathfrak{f}^{(n)}\|_{H^2(\Omega)}+\|\mathfrak{g}^{(n)}\|_{H^3({\Gamma_{\rm en}})}\right),\\
&\|w\|_{H^4_{\ast}(\Omega)}\le C\left(\|\mathcal{F}^{(n)}\|_{H^2(\Omega)}+\|\mathfrak{f}^{(n)}\|_{H^2(\Omega)}+\|\mathfrak{g}^{(n)}\|_{H^2({\Gamma_{\rm en}})}\right)
\end{split}
\end{equation}
for a constant $C>0$ depending only on the data and $L$.
\end{lemma}
%
For reader's convenience, the proof of Lemma \ref{Lemma27} is given later since it is not short. Assume that Lemma \ref{Lemma27} is true.


{\bf Step 2.} The well-posedness of \eqref{fix-lin-ho}: 
For each $n\in\mathbb{N}$, let $(v_n,w_n)\in [H^4_{\ast}(\Omega)]^2$ be the unique axisymmetric solution to \eqref{app-vw} obtained in Lemma \ref{Lemma27}. 
Then, by \eqref{vw0est} and the general Sobolev inequality, the sequences $\{(v_n,w_n)\}_{n\in\mathbb{N}}$ and $\{(\partial_rv_n,\partial_rw_n)\}_{n\in\mathbb{N}}$ are uniformly bounded in $C^{1,\frac{1}{2}}(\overline{\Omega})$.
The Arzel\'a-Ascoli theorem and the weak compactness property of $H^3(\Omega)$ imply there exist a subsequence $\{(v_{n_k},w_{n_k})\}$ and functions $(\chi_h,\Psi_h)$ so that 
\begin{equation*}
\begin{split}
&(v_{n_k},w_{n_k})\to(\chi_h,\Psi_h)\mbox{ and }(\partial_rv_{n_k},\partial_rw_{n_k})\to(\partial_r\chi_h,\partial_r\Psi_h)\mbox{ in }[ C^{1,\frac{1}{4}}(\overline{\Omega})]^2,\\
&(v_{n_k},w_{n_k})\rightharpoonup(\chi_h,\Psi_h)\mbox{ and }(\partial_rv_{n_k},\partial_rw_{n_k})\rightharpoonup(\partial_r\chi_h,\partial_r\Psi_h)\mbox{ in }[H^3(\Omega)]^2.
\end{split}
\end{equation*}
Then $(\chi_h,\Psi_h)$ is an axisymmetric strong solution to \eqref{fix-lin-ho} and satisfies 
\begin{equation}\label{H3-1}
\begin{split}
&\|\chi_h\|_{H^4_{\ast}(\Omega)}\le C\left(\|\mathcal{F}_h\|_{H^3_{\ast}(\Omega)}+\|\mathfrak{f}_h\|_{H^2(\Omega)}+\|\mathfrak{g}_h\|_{H^3({\Gamma_{\rm en}})}\right),\\
&\|\Psi_h\|_{H^4_{\ast}(\Omega)}\le C\left(\|\mathcal{F}_h\|_{H^2(\Omega)}+\|\mathfrak{f}_h\|_{H^2(\Omega)}+\|\mathfrak{g}_h\|_{H^2({\Gamma_{\rm en}})}\right)
\end{split}
\end{equation} 
for a constant $C$ depending only on the data and $L$.
One can easily check that the solution is unique by using a standard contraction argument if ${\bm\tau}_p\in(0,\tilde{\bm\tau}_{\flat}]$ for a sufficiently small constant $\tilde{\bm\tau}_\flat\in(0,{\bm\tau}^{\flat}]$ depending only on the data and $(L,\delta_1,\delta_2,\delta_3)$.
The proof of Lemma \ref{Lem-po} is completed.
\end{proof}

Now we prove Lemma \ref{Lemma27}.
\begin{proof}[Proof of Lemma \ref{Lemma27}]
The proof is divided into four steps.

{\bf Step 1.} Approximate Problems:
%
Let $\{\xi_k=\xi_k(\phi)\}_{k=0}^{\infty}$ be eigenfunctions with eigenvalues $\{\omega_k\}_{k=0}^{\infty}$ to the eigenvalue problem 
\begin{equation}\label{eigen}
\left\{\begin{split}
-\frac{1}{\sin\phi}(\sin\phi \xi')'=\omega\xi\quad&\mbox{ in }(0,\phi_0), \\
\xi'=0\quad&\mbox{ on }\{\phi=0,\phi_0\}.
\end{split}\right.
\end{equation}
Then
$\{\xi_k\}_{k=0}^{\infty}$ forms an orthonormal basis in $L^2((0,\phi_0);\sin\phi d\phi)$ with respect to the inner product defined by 
\begin{equation}\label{inner}
\langle\xi_j,\xi_k\rangle=\int_0^{\phi_0}\xi_j(\phi)\xi_k(\phi)\sin\phi\, d{\phi}.
\end{equation}

Let $(V_m,W_m)$ be of the form
\begin{equation}\label{VW-form}
V_m(r,\phi):=\sum_{j=0}^{m}v_j(r)\xi_j(\phi),
\quad 
W_m(r,\phi):=\sum_{j=0}^{m}w_j(r)\xi_j(\phi)
\end{equation}
for $(r,\phi)\in\mathcal{R}_L:=\{(r,\phi)\in\mathbb{R}^2: r_{\rm en}<r<r_{\rm ex},\,0<\phi<\phi_0\}$.
For each $m\in\mathbb{N}$, we seek a solution $(V_m, W_m)$ satisfying
\begin{equation}\label{VWMM}
\begin{split}
&\langle \mathcal{L}_1^{(n)}(V_m,W_m),\xi_k\rangle=\langle\mathcal{F}^{(n)},\xi_k \rangle,\\
&\langle\mathscr{L}_2(V_m,W_m),\xi_k\rangle=\langle \mathfrak{f}^{(n)},\xi_k\rangle\mbox{ for all }k=0,1,\ldots, m
\end{split}
\end{equation}
with boundary conditions
\begin{equation}\label{VW-BD}
\left\{\begin{split}
V_m=0,\,\, \partial_r V_m=\sum_{j=0}^{m}\langle \mathfrak{g}^{(n)},\xi_j\rangle\xi_j,\,\,\partial_r W_m=0\mbox{ on }\Gamma_{\rm en},\\
\partial_{\phi}V_m=0,\,\, \partial_{\phi}W_m=0\mbox{ on }\Gamma_{\rm w},\\
\partial_r W_m=0\mbox{ on }\Gamma_{\rm ex}.
\end{split}\right.
\end{equation}
%
%
As in the proof of Lemma \ref{lem-H1}, we get 
\begin{equation}\label{VW-H1}
\begin{split}
\|V_m\|_{H^1(\Omega)}&+\|W_m\|_{H^1(\Omega)}
\le C\left(\|\mathcal{F}^{(n)}\|_{L^2(\Omega)}+\|\mathfrak{f}^{(n)}\|_{L^2(\Omega)}+\|\mathfrak{g}^{(n)}\|_{L^2({\Gamma_{\rm en}})}\right)
\end{split}
\end{equation}
for a constant $C>0$ depending only on the data and $L$.

{\bf Step 2.} Analysis on ODE problems:  
We  rewrite \eqref{VWMM}-\eqref{VW-BD} as a boundary value problem for $\{(v_k,w_k)\}_{k=0}^{m}$ as follows:
\begin{equation}\label{vw-kk}
\left\{\begin{split}
v_k''+2\langle a_{12}^{(n)}\sum_{j=0}^m\partial_{\phi}\xi_j,\xi_k\rangle v_k'+\langle a_{22}^{(n)}\sum_{j=0}^m\partial_{\phi}^2\xi_j,\xi_k\rangle v_k
+ (a_1^{(n)}-\alpha_1)v_k'\\
+\langle a_2^{(n)}\sum_{j=0}^m\partial_{\phi}\xi_j,\xi_k\rangle v_k-b_1w_k'-cw_k=\langle\mathcal{F}^{(n)},\xi_k \rangle,\\
w_k''+\sum_{j=0}^m\langle \frac{1}{r^2}\partial_{\phi}^2\xi_j,\xi_k\rangle w_k+\frac{2}{r}w_k'+\sum_{j=0}^m\langle \frac{\cos\phi}{r^2\sin\phi}\partial_{\phi}\xi_j,\xi_k\rangle w_k\\
-g_0 w_k-h_1 v_k'=\langle \mathfrak{f}^{(n)},\xi_k\rangle,\\
v_k(r_{\rm en})=0,\, v_k'(r_{\rm en})=\langle \mathfrak{g}^{(n)},\xi_k\rangle, \,w_k'(r_{\rm en})=0,\, w_k'(r_{\rm ex})=0.
\end{split}\right.
\end{equation}
By using the Fredholm alternative theorem and the bootstrap argument, one can prove that 
the boundary value problem \eqref{vw-kk} for $\{(v_k,w_k)\}_{k=0}^m$ has a unique smooth solution on the interval $[r_{\rm en},r_{\rm ex}]$.

{\bf Step 3.} Estimate for $(V_m,W_m)$:
\begin{lemma}\label{lem28}
Fix $n\in\mathbb{N}$.
For each $m\in\mathbb{N}$, let $(V_m,W_m)$ given in the form \eqref{VW-form} be the solution to \eqref{VWMM}-\eqref{VW-BD}. Then $(V_m,W_m)$ satisfies  
\begin{equation}\label{lem28-est}
\begin{split}
&\|V_m\|_{H^4_{\ast}(\Omega)}\le C\left(\|\mathcal{F}^{(n)}\|_{H^3_{\ast}(\Omega)}+\|\mathfrak{f}^{(n)}\|_{H^2(\Omega)}+\|\mathfrak{g}^{(n)}\|_{H^3({\Gamma_{\rm en}})}\right),\\
&\|W_m\|_{H^4_{\ast}(\Omega)}\le C\left(\|\mathcal{F}^{(n)}\|_{H^2(\Omega)}+\|\mathfrak{f}^{(n)}\|_{H^2(\Omega)}+\|\mathfrak{g}^{(n)}\|_{H^2({\Gamma_{\rm en}})}\right)
\end{split}
\end{equation}
for a constant $C>0$ depending only on the data and $L$.
The estimate constant $C$ is independent of $n,m\in\mathbb{N}$ and $(S_{\ast},\frac{\Lambda_{\ast}}{r\sin\phi}{\bf e}_{\theta},\chi_{\ast},\Psi_{\ast},{\bf W}_{\ast})\in\mathcal{J}(\delta_1)\times\mathcal{J}(\delta_2)\times\mathcal{J}(\delta_3)$.
\end{lemma}

From now on, the constant $C$ depends only on the data and $L$, which may vary from line to line.

\begin{proof}[Proof of Lemma \ref{lem28}]
The proof is divided into three steps. 

{\bf Step 1.} ($H^2$ estimate of $(V_m,W_m)$)
Fix $m\in\mathbb{N}$ and set 
\begin{equation*}
\mathcal{F}_m:=\sum_{j=0}^m\langle\mathcal{F}^{(n)},\xi_j\rangle\xi_j,\quad \mathfrak{f}_m:=\sum_{j=0}^m\langle\mathfrak{f}^{(n)},\xi_j\rangle\xi_j.
\end{equation*}
Then it follows from \eqref{VWMM}-\eqref{VW-BD} that
\begin{equation}\label{Wm-eq}
\left\{\begin{split}
\partial_{rr}W_m+\frac{1}{r^2}\partial_{\phi\phi}W_m-g_0W_m=\mathfrak{f}_m^{\ast}\mbox{ in }\Omega,\\
\partial_rW_m=0\mbox{ on }\Gamma_{\rm en},\quad \partial_r W_m=0\mbox{ on }\Gamma_{\rm ex},\quad \partial_{\phi}W_m=0\mbox{ on }\Gamma_{\rm w}
\end{split}
\right.
\end{equation}
for 
\begin{equation}\label{fmstar}
\mathfrak{f}_{m}^{\ast}:=-\frac{2}{r}\partial_rW_m-\frac{\cos\phi}{r^2\sin\phi}\partial_{\phi}W_m+h_1\partial_r V_m+\mathfrak{f}_m.
\end{equation}
By using the method of reflection across the wall, we get
\begin{equation*}
\|W_m\|_{H^2(\Omega)}\le C\left(\|\mathfrak{f}^{(n)}\|_{L^2(\Omega)}+\|W_m\|_{H^1(\Omega)}+\|V_m\|_{H^1(\Omega)}\right).
\end{equation*}
It follows from \eqref{VW-H1} that 
\begin{equation}\label{W-H2}
\|W_m\|_{H^2(\Omega)}\le C\left(\|\mathcal{F}^{(n)}\|_{L^2(\Omega)}+\|\mathfrak{f}^{(n)}\|_{L^2(\Omega)}+\|\mathfrak{g}^{(n)}\|_{L^2({\Gamma_{\rm en}})}\right).
\end{equation}
 
Now we compute a $H^2$ estimate of $V_m$. From \eqref{VWMM}, we have
\begin{equation*}
\langle \mathcal{L}_1^{(n)}(V_m,W_m),\xi_k\rangle=\langle\mathcal{F}_m,\xi_k \rangle,\quad k=0,1,\ldots,m.
\end{equation*}
This can be rewritten as 
\begin{equation}\label{643}
\langle \mathcal{L}_{h}(V_m),\xi_k\rangle=\langle\mathfrak{F}_m+b_1\partial_rW_m+cW_m,\xi_k\rangle,\quad k=0,1,\ldots,m
\end{equation}
for $\mathcal{L}_h(V):=\sum_{i,j=1}^2 a_{ij}^{(n)}\partial_{ij}V+\sum_{i=1}^2(a_i^{(n)}-\alpha_i)\partial_iV$.
For notational simplicity, set 
$$q_m:=\partial_rV_m.$$
Differentiating \eqref{643} with respect to $r$ gives
\begin{equation}\label{qm-eq}
\begin{split}
\langle \mathcal{L}_{h}(q_m),\xi_k\rangle
=&\langle\partial_r(\mathfrak{F}_m+b_1\partial_rW_m+cW_m),\xi_k\rangle\\
&+\langle-2\partial_ra_{12}^{(n)}\partial_{\phi}q_m-\partial_ra_{22}^{(n)}\partial_{\phi}^2V_m-\sum_{i=1}^2\partial_r(a_i^{(n)}-\alpha_i)q_m,\xi_k\rangle.
\end{split}
\end{equation}
We multiply \eqref{qm-eq} by $v_k''$ for each $k=0,1,\ldots,m$, and add up the results over $k=0$ to $m$, finally integrate the summation with respect to $r$ on the interval $[r_{\rm en}, t]$ for $t$ varying in the interval $[r_{\rm en},r_{\rm ex}]$ to get 
\begin{equation}\label{A7}
\begin{split}
\int_{\mathcal{N}_t}\mathcal{L}_h(q_m)&\partial_r q_m (r^2\sin\phi) d{\bf y}\\
&=\int_{\mathcal{N}_t}\partial_r\left(\mathfrak{F}_m+b_1\partial_rW_m+cW_m\right)\partial_rq_m(r^2\sin\phi)d{\bf y}\\
&\quad+\int_{\mathcal{N}_t} \left(-2\partial_ra_{12}^{(n)}\partial_{\phi}q_m-\partial_ra_{22}^{(n)}\partial_{\phi}^2V_m\right)\partial_rq_m(r^2\sin\phi)d{\bf y}\\
&\quad+\int_{\mathcal{N}_t} \left(-\sum_{i=1}^2\partial_r(a_i^{(n)}-\alpha_i)q_m\right)\partial_rq_m(r^2\sin\phi)d{\bf y}
\end{split}
\end{equation}
for $\mathcal{N}_t:=\mathcal{N}_L\cap\{r<t\}$. 
By the definition of $\mathcal{L}_h(V)$, we obtain that 
\begin{equation*}
\begin{split}
&\mbox{LHS of \eqref{A7}}\\
&=\frac{1}{2}\left(\int_{\overline{\mathcal{N}_t}\cap\{r=t\}}-\int_{\overline{\mathcal{N}_t}\cap\{r=r_{\rm en}\}}\right)\left[(\partial_r q_m)^2-a_{22}^{(n)}(\partial_{\phi}q_m)^2\right](r^2\sin\phi)d\phi\\
&\quad-\int_{\mathcal{N}_t}\left[2r\sin\phi+2(\partial_{\phi}a_{12}^{(n)})r^2\sin\phi+2a_{12}^{(n)}r^2\cos\phi \right]\frac{(\partial_rq_m)^2}{2} d{\bf y}\\
&\quad-\int_{\mathcal{N}_t}\left[(\partial_{\phi}a_{22}^{(n)})(r^2\sin\phi)+a_{22}^{(n)}(r^2\cos\phi)\right](\partial_rq_m)(\partial_{\phi}q_m)d{\bf y}\\
&\quad+\int_{\mathcal{N}_t}\left[(\partial_ra_{22}^{(n)})r^2\sin\phi+2a_{22}^{(n)}r\sin\phi\right]\frac{(\partial_{\phi}q_m)^2}{2}d{\bf y}\\
&\quad+\int_{\mathcal{N}_t}\sum_{i=1}^2(a_i^{(n)}-\alpha_i)(\partial_i q_m)(\partial_rq_m)(r^2\sin\phi)d{\bf y}.
\end{split}
\end{equation*}
%
Since
\begin{equation*}
\begin{split}
\left|\frac{a_{12}^{(n)}}{\sin\phi}\right|+\left|\frac{a_{22}^{(n)}}{\sin\phi}\right|\le C,
\end{split}
\end{equation*}
 there exist positive constants $C_i$ ($i=1,2,3$) depending only on the data and $L$ such that 
\begin{equation}\label{A8}
\begin{split}
\mbox{LHS of \eqref{A7}}
\ge&\, C_1\int_{\overline{\mathcal{N}_t}\cap\{r=t\}}|Dq_m|^2(r^2\sin\phi)d\phi\\
&-C_2\int_{\overline{\mathcal{N}_t}\cap\{r=r_{\rm en}\}}|Dq_m|^2 (r^2\sin\phi)d\phi\\
&-C_3\int_{\mathcal{N}_t}|Dq_m|^2 (r^2\sin\phi)d{\bf y}.
\end{split}
\end{equation}
for $D:=\nabla_{(r,\phi)}$. 
It follows from \eqref{A7}-\eqref{A8} that 
\begin{equation}\label{Gron}
\begin{split}
&\int_{\overline{\mathcal{N}_t}\cap\{r=t\}}|Dq_m|^2(r^2\sin\phi) d\phi\\
&\le\, \frac{C_2}{C_1}\int_{\overline{\mathcal{N}_t}\cap\{r=r_{\rm en}\}}|Dq_m|^2 (r^2\sin\phi)d\phi
+ C\int_{\mathcal{N}_t}(|Dq_m|^2+|\partial_{\phi}^2V_m|^2)(r^2\sin\phi) d{\bf y}\\
&\quad+C\left(\|\mathcal{F}^{(n)}\|_{H^1(\Omega)}+\|\mathfrak{f}^{(n)}\|_{L^2(\Omega)}+\|\mathfrak{g}^{(n)}\|_{H^1({\Gamma_{\rm en}})}\right).
\end{split}
\end{equation}

Now we need to compute estimates of 
\begin{equation*}
\mbox{(i) }\int_{\overline{\mathcal{N}_t}\cap\{r=r_{\rm en}\}}|Dq_m|^2 (r^2\sin\phi)d\phi\quad\mbox{and}\quad
\mbox{(ii) }\int_{\mathcal{N}_t}|\partial_{\phi}^2V_m|^2(r^2\sin\phi)d{\bf y}.
\end{equation*}
First we compute (i). We have 
\begin{equation*}
\begin{split}
&\int_{\overline{\mathcal{N}_t}\cap\{r=r_{\rm en}\}}|\partial_{\phi}q_m|^2 (r^2\sin\phi)d\phi\\
&=r_{\rm en}^2\int_{\overline{\mathcal{N}_t}\cap\{r=r_{\rm en}\}}\left(\sum_{j=0}^m\langle\mathfrak{g}^{(n)},\xi_j \rangle \xi_j'(\phi)\right)^2(\sin\phi) d\phi
=r_{\rm en}^2 \sum_{j=0}^m\langle\mathfrak{g}^{(n)},\xi_j \rangle^2\omega_j\\
\end{split}
\end{equation*}
since
\begin{equation*}
\int_{\overline{\mathcal{N}_t}\cap\{r=r_{\rm en}\}}\langle\mathfrak{g}^{(n)},\xi_j \rangle\langle\mathfrak{g}^{(n)},\xi_k \rangle \xi_j'\xi_k' \sin\phi d\phi
=\left\{\begin{split}
	\langle\mathfrak{g}^{(n)},\xi_j \rangle^2\omega_j\quad &\mbox{for }j=k,\\
	0\quad&\mbox{otherwise.}
	\end{split}\right.
\end{equation*}
It follows from
\begin{equation}\label{664}
\begin{split}
\langle \mathfrak{g}^{(n)},\xi_j\rangle
&=\int_0^{\phi_0}\mathfrak{g}^{(n)}\xi_j\sin\phi d\phi=-\frac{1}{\omega_j}\int_0^{\phi}\mathfrak{g}^{(n)}(\sin\phi \xi_j')'d\phi\\
&=\frac{1}{\omega_j}\int_0^{\phi}(\mathfrak{g}^{(n)})'\sin\phi\xi_j' d\phi
=\frac{1}{\sqrt{\omega_j}}\langle(\mathfrak{g}^{(n)})',\frac{\xi_j'}{\sqrt{\omega_j}}\rangle
\end{split}
\end{equation}
that
\begin{equation*}
\begin{split}
\int_{\overline{\mathcal{N}_t}\cap\{r=r_{\rm en}\}}|\partial_{\phi}q_m|^2 (r^2\sin\phi)d\phi
=r_{\rm en}^2 \sum_{j=0}^m\langle(\mathfrak{g}^{(n)})',\frac{\xi_j'}{\sqrt{\omega_j}}\rangle^2.
\end{split}
\end{equation*}
Since the set $\{\frac{\xi_j'}{\sqrt{\omega_j}}\}_{j=1}^{\infty}$ forms an orthonormal basis in $L^2((0,\phi_0);\sin\phi d\phi)$, we have 
\begin{equation*}
\int_{\overline{\mathcal{N}_t}\cap\{r=r_{\rm en}\}}|\partial_{\phi} q_m|^2 (r^2\sin\phi)d\phi
\le \phi_0 r_{\rm en}^2 \|\mathfrak{g}^{(n)}\|_{H^1({\Gamma_{\rm en}})}^2.
\end{equation*}
For each $k=0,1,\ldots, m$,  multiplying \eqref{643} by $v_k''$ and summing up over $k=0$ to $m$ gives
\begin{equation*}
\begin{split}
&\int_{\overline{\mathcal{N}_t}\cap\{r=r_{\rm en}\}}\left(|\partial_rq_m|^2+2a_{12}^{(n)}(\partial_{\phi}q_m)(\partial_rq_m)+(a_1^{(n)}-\alpha_1)q_m\partial_rq_m\right)(\sin\phi) d\phi\\
&=\int_{\overline{\mathcal{N}_t}\cap\{r=r_{\rm en}\}}\left(\mathcal{F}^{(n)}+cW_m\right)\partial_rq_m (\sin\phi)d\phi.
\end{split}
\end{equation*}
By the Cauchy-Schwartz inequality, trace inequality, and \eqref{VW-H1}, we get 
\begin{equation}\label{H22-ent}
\begin{split}
&\int_{\overline{\mathcal{N}_t}\cap\{r=r_{\rm en}\}}|Dq_m|^2(r^2\sin\phi)d{\phi}\\
&\le C\left(\|\mathcal{F}^{(n)}\|_{H^1(\Omega)}+\|\mathfrak{f}^{(n)}\|_{L^2(\Omega)}+\|\mathfrak{g}^{(n)}\|_{H^1({\Gamma_{\rm en}})}\right).
\end{split}
\end{equation}

Now we compute (ii). 
For $k=0,1,\ldots,m$,
\eqref{643} can be rewritten as 
\begin{equation*}
\begin{split}
\langle a_{22}^{(n)}\partial_{\phi}^2V_m,\xi_k\rangle
=\langle& \mathfrak{R},\xi_k\rangle,
\end{split}
\end{equation*}
where $\mathfrak{R}$ is defined by 
\begin{equation*}
\mathfrak{R}:=-\partial_rq_m-2a_{12}^{(n)}\partial_\phi q_m-(a_1^{(n)}-\alpha_1)q_m-a_2^{(n)}\partial_\phi V_m
-\mathfrak{F}_m+b_1\partial_rW_m+cW_m.
\end{equation*}
It follows from \eqref{eigen} that 
\begin{equation}\label{xi2}
\langle a_{22}^{(n)}\partial_{\phi}^2V_m, \xi_k''\rangle=\langle \mathfrak{R}, \frac{\cos\phi}{\sin\phi}\xi_k'+\xi_k''\rangle-\langle a_{22}^{(n)}\partial_{\phi}^2V_m, \frac{\cos\phi}{\sin\phi}\xi_k'\rangle.
\end{equation}
By the definition of $\langle\cdot,\cdot\rangle$ and \eqref{xi2}, we have
\begin{equation*}
\begin{split}
&-\int_0^{\phi_0}a_{22}^{(n)}(\partial_{\phi}^2V_m)^2\sin\phi d\phi\\
&=-\sum_{k=0}^mv_k\int_0^{\phi_0}a_{22}^{(n)}\partial_{\phi}^2V_m \xi_k'' \sin\phi d\phi\\
&=-\sum_{k=0}^m v_k\int_0^{\phi_0}\mathfrak{R}\left( \frac{\cos\phi}{\sin\phi}\xi_k'+\xi_k''\right)\sin\phi d\phi+\sum_{k=0}^mv_k\int_0^{\phi_0}a_{22}^{(n)}\partial_{\phi}^2V_m\frac{\cos\phi}{\sin\phi} \xi_k'\sin\phi d\phi\\
&=-\int_0^{\phi_0}\mathfrak{R}(\partial_\phi V_m)\cos\phi+\mathfrak{R}(\partial_{\phi}^2V_m)\sin\phi d\phi+\int_0^{\phi_0}a_{22}^{(n)}(\partial_{\phi}^2V_m)(\partial_{\phi} V_m)\cos\phi d\phi\\
&=:-(A_1+A_2)+A_3.
\end{split}
\end{equation*}
By the Cauchy-Schwartz inequality and the axisymmetric properties, we have
\begin{equation*}
\begin{split}
|A_1|+|A_2|+|A_3|
\le&\,C\int_0^{\phi_0}\left(\mathfrak{R}^2+|\partial_{\phi}V_m|^2\right)\sin\phi d\phi\\
&+\epsilon\int_0^{\phi_0}|\partial_{\phi}^2V_m|^2\sin\phi d\phi
\end{split}
\end{equation*}
for a sufficiently small constant $\epsilon\in(0,\inf_{{\bf y}\in\overline{\mathcal{N}_L}}\left\{1,\frac{-a_{22}^{(n)}({\bf y})}{4}\right\})$ depending only on the data.
Then
\begin{equation}\label{phi22}
\int_{\mathcal{N}_t}|\partial_\phi^2V_m|^2 (r^2\sin\phi) d{\bf y}\le C\left(\|\mathcal{F}^{(n)}\|_{H^1(\Omega)}+\|\mathfrak{f}^{(n)}\|_{L^2(\Omega)}+\|\mathfrak{g}^{(n)}\|_{H^1({\Gamma_{\rm en}})}\right).
\end{equation}
It follows from \eqref{Gron}, \eqref{H22-ent}, and \eqref{phi22} that 
\begin{equation*}
\begin{split}
&\int_{\overline{\mathcal{N}_t}\cap\{r=t\}}|Dq_m|^2(r^2\sin\phi) d\phi\\
&\le C\int_{\mathcal{N}_t}|Dq_m|^2(r^2\sin\phi) d{\bf y}
+C\left(\|\mathcal{F}^{(n)}\|_{H^1(\Omega)}+\|\mathfrak{f}^{(n)}\|_{L^2(\Omega)}+\|\mathfrak{g}^{(n)}\|_{H^1({\Gamma_{\rm en}})}\right).
\end{split}
\end{equation*}
Applying the Gronwall's inequality with \eqref{phi22} gives
\begin{equation*}
\|V_m\|_{H^2(\Omega)}\le C\left(\|\mathcal{F}^{(n)}\|_{H^1(\Omega)}+\|\mathfrak{f}^{(n)}\|_{L^2(\Omega)}+\|\mathfrak{g}^{(n)}\|_{H^1({\Gamma_{\rm en}})}\right).
\end{equation*}

{\bf Step 2.} ($H^3$ estimate of $(V_m,W_m)$)
Back to \eqref{Wm-eq}.
Since $\mathfrak{f}_m^{\ast}$ is $H^1$ near the wall $\Gamma_{\rm w}$, we again apply the method of reflection across the wall $\Gamma_{\rm w}$ to get 
\begin{equation}\label{Wm-H3}
\|W_m\|_{H^3(\Omega)}\le 
C\left(\|\mathcal{F}^{(n)}\|_{H^1(\Omega)}+\|\mathfrak{f}^{(n)}\|_{H^1(\Omega)}+\|\mathfrak{g}^{(n)}\|_{H^1({\Gamma_{\rm en}})}\right).
\end{equation}
%
By using \eqref{664} and  $(\mathfrak{g}^{(n)})'(\phi_0)=0$, we have
\begin{equation}\label{ggnn}
\begin{split}
\langle\mathfrak{g}^{(n)},\xi_j\rangle
&=\left\langle \mathfrak{g}^{(n)},-\frac{(\sin\phi\xi_j')'}{\omega_j\sin\phi}\right\rangle
=\frac{1}{\omega_j}\left\langle (\mathfrak{g}^{(n)})',\xi_j'\right\rangle\\
&=-\frac{1}{\omega_j}\langle (\mathfrak{g}^{(n)})'',\xi_j\rangle-\frac{1}{\omega_j}\left\langle \frac{(\mathfrak{g}^{(n)})'\cos\phi}{\sin\phi},\xi_j\right\rangle.
\end{split}
\end{equation}
Since the equation in \eqref{eigen} is equivalent to 
\begin{equation}\label{delta-xi}
-\Delta_{\bf x}\xi=\frac{\omega}{r^2}\xi,
\end{equation}
it follows from \eqref{ggnn} and \eqref{delta-xi} that 
\begin{equation*}
\begin{split}
\langle \mathfrak{g}^{(n)},\xi_j\rangle\Delta_{\bf x}\xi_j
&=-\langle \mathfrak{g}^{(n)},\xi_j\rangle\frac{\omega_j}{r^2}\xi_j\\
&=\frac{1}{r^2}\langle (\mathfrak{g}^{(n)})'',\xi_j\rangle\xi_j+\frac{1}{r^2}\left\langle \frac{(\mathfrak{g}^{(n)})'\cos\phi}{\sin\phi},\xi_j\right\rangle\xi_j.
\end{split}
\end{equation*}
This implies that 
\begin{equation}\label{gc2}
\begin{split}
\int_{\Gamma_{\rm en}}(\partial_{\phi}^2\partial_r V_m)^2 \sin\phi d\phi
&=\int_{\Gamma_{\rm en}}\left(\sum_{j=0}^m\langle \mathfrak{g}^{(n)},\xi_j\rangle\xi_j''\right)^2\sin\phi d\phi\\
&\le C \|\mathfrak{g}^{(n)}\|_{H^2({\Gamma_{\rm en}})}^2.
\end{split}
\end{equation}
Similarly, since it holds that 
$\partial_\phi\mathcal{F}^{(n)}\equiv 0\mbox{ on }\Gamma_{\rm w},$ we have
\begin{equation}\label{FH2}
\|\mathcal{F}_m\|_{H^2(\Omega)}\le C \|\mathcal{F}^{(n)}\|_{H^2(\Omega)}.
\end{equation}
Applying the same way in Step 1 and using \eqref{gc2}-\eqref{FH2} imply
\begin{equation}\label{Vm-H3}
\|V_m\|_{H^3(\Omega)}\le C\left(\|\mathcal{F}^{(n)}\|_{H^2(\Omega)}+\|\mathfrak{f}^{(n)}\|_{H^1(\Omega)}+\|\mathfrak{g}^{(n)}\|_{H^2({\Gamma_{\rm en}})}\right).
\end{equation}

{\bf Step 3.} ($H^3$ estimate of $(\partial_rV_m,\partial_rW_m)$)
Differentiating \eqref{Wm-eq} with respect to $r$ gives 
\begin{equation*}
\left\{\begin{split}
\partial_{rr}(\partial_rW_m)+\frac{1}{r^2}\partial_{\phi\phi}(\partial_rW_m)-g_0(\partial_rW_m)=\mathfrak{R}_m\mbox{ in }\Omega,\\
\partial_rW_m=0\mbox{ on }\Gamma_{\rm en}\cup\Gamma_{\rm ex},\,\, \partial_{\phi}(\partial_rW_m)=0\mbox{ on }\Gamma_{\rm w}
\end{split}
\right.
\end{equation*}
for $\mathfrak{R}_m$ given by 
\begin{equation*}
\mathfrak{R}_m:=\frac{2}{r^3}\partial_{\phi\phi}W_m+g_0'W_m+\partial_r\mathfrak{f}_m^{\ast}\in H^1(\Omega).
\end{equation*}
%
As in Step 2, we use the method of reflection to get the estimate 
\begin{equation*}
\begin{split}
\|\partial_rW_m\|_{H^3(\Omega)}
\le &\,C\left(\|W_m\|_{H^3(\Omega)}+\|\partial_r\mathfrak{f}_m^{\ast}\|_{H^1(\Omega)}\right).
\end{split}
\end{equation*}
By  the definition of $\mathfrak{f}_m^{\ast}$ given in \eqref{fmstar}, we have
\begin{equation}\label{484}
\begin{split}
\|\partial_rW_m\|_{H^3(\Omega)}
\le &\,C\left(\|V_m\|_{H^3(\Omega)}+\|W_m\|_{H^3(\Omega)}+\|\partial_r\mathfrak{f}_m\|_{H^1(\Omega)}\right).
\end{split}
\end{equation}
By \eqref{Wm-H3}, \eqref{Vm-H3}, and \eqref{484}, we have 
\begin{equation*}
\begin{split}
&\|\partial_rW_m\|_{H^3(\Omega)}\\
&\le \,C\left(\|\mathcal{F}^{(n)}\|_{H^2(\Omega)}+\|\mathfrak{f}^{(n)}\|_{H^1(\Omega)}+\|\mathfrak{g}^{(n)}\|_{H^2({\Gamma_{\rm en}})}+\|\partial_r\mathfrak{f}^{(n)}\|_{H^1(\Omega)}\right).
\end{split}
\end{equation*}


To compute a $H^3$ estimate of $q_m(=\partial_rV_m)$, we note that 
\begin{equation}\label{omega-ab}
-\Delta_{\bf x}(\partial_{x_i}\xi)
=\omega\partial_{x_i}\left(\frac{\xi}{r^2}\right)=\omega a_i\xi+\omega b_i\xi'
\end{equation}
for $a_i$ and $b_i$ ($i=1,2,3$) given by 
\begin{equation*}
\left\{\begin{split}
&a_1:=-\frac{2}{r^3}\sin\phi\cos\theta,\quad b_1:=\frac{1}{r^3}\cos\phi\cos\theta,\\
&a_2:=-\frac{2}{r^3}\sin\phi\sin\phi,\quad b_2:=\frac{1}{r^3}\cos\phi\sin\theta,\\
&a_3:=-\frac{2}{r^3}\cos\phi,\quad b_3:=-\frac{1}{r^3}\sin\phi.
\end{split}\right.
\end{equation*}
Then, by \eqref{664}, \eqref{ggnn} and \eqref{omega-ab}, we have
\begin{equation}\label{xi-eq1}
\begin{split}
\langle \mathfrak{g}^{(n)},\xi_j\rangle\Delta_{\bf x}(\partial_{x_i}\xi_j)
&=-\langle \mathfrak{g}^{(n)},\xi_j\rangle\omega_j a_i\xi_j-\langle \mathfrak{g}^{(n)},\xi_j\rangle\omega_j b_i\xi_j'\\
&=\langle (\mathfrak{g}^{(n)})'',\xi_j\rangle a_i\xi_j+\left\langle\frac{(\mathfrak{g}^{(n)})'\cos\phi}{\sin\phi},\xi_j
\right\rangle a_i\xi_j\\
&\quad+\langle (\mathfrak{g}^{(n)})'',\xi_j\rangle b_i\xi_j'+\left\langle\frac{(\mathfrak{g}^{(n)})'\cos\phi}{\sin\phi},\xi_j\right\rangle b_i\xi_j'.
\end{split}
\end{equation}
Similarly to \eqref{664}, we get
\begin{equation}\label{xi-eq2}
\begin{split}
&\langle (\mathfrak{g}^{(n)})'',\xi_j\rangle b_i\xi_j'+\left\langle\frac{(\mathfrak{g}^{(n)})'\cos\phi}{\sin\phi},\xi_j\right\rangle b_i\xi_j'\\
&=\frac{1}{\sqrt{\omega_j}}\left\langle (\mathfrak{g}^{(n)})''',\frac{\xi_j'}{\sqrt{\omega_j}}\right\rangle b_i\xi_j'+\frac{1}{\sqrt{\omega_j}}\left\langle\left(\frac{(\mathfrak{g}^{(n)})'\cos\phi}{\sin\phi}\right)',\frac{\xi_j'}{\sqrt{\omega_j}}\right\rangle b_i\xi_j'.
\end{split}
\end{equation}
By using \eqref{xi-eq1}-\eqref{xi-eq2}, we get 
\begin{equation*}
\begin{split}
\int_{\Gamma_{\rm en}}(\partial_{\phi}^3\partial_r V_m)^2 \sin\phi d\phi
&=\int_{\Gamma_{\rm en}}\left(\sum_{j=0}^m\langle \mathfrak{g}^{(n)},\xi_j\rangle\xi_j'''\right)^2\sin\phi d\phi\\
&\le C \|\mathfrak{g}^{(n)}\|_{H^3({\Gamma_{\rm en}})}^2.
\end{split}
\end{equation*}
Applying the method to obtain the $H^3$ estimate of $V_m$ in Steps 1 and 2 gives
\begin{equation*}
\|\partial_rV_m\|_{H^3(\Omega)}\le C\left(\|\partial_r\mathcal{F}^{(n)}\|_{H^2(\Omega)}+\|\partial_r\mathfrak{f}^{(n)}\|_{H^1(\Omega)}+\|\mathfrak{g}^{(n)}\|_{H^3({\Gamma_{\rm en}})}\right).
\end{equation*}
The proof of Lemma \ref{lem28} is completed.\end{proof}

{\bf Step 4.} Taking the limit: 
For a fixed $n\in\mathbb{N}$, let $(V_m,W_m)$ given in the form \eqref{VW-form} be the solution to \eqref{VWMM}-\eqref{VW-BD}. 
By Lemma \ref{lem28}, the solution satisfies \eqref{lem28-est}.
The general Sobolev inequality implies that $\{(V_m,W_m)\}_{m\in\mathbb{N}}$ and $\{(\partial_rV_m,\partial_rW_m)\}_{m\in\mathbb{N}}$ are bounded in $[C^{1,\frac{1}{2}}(\overline{\Omega})]^2$.
Then there exists a subsequence $\{(V_{m_k},W_{m_k})\}_{k\in\mathbb{N}}$ so that the limit $(v,w)\in [H^4_{\ast}({\Omega})\cap C^{1,\frac{1}{2}}(\overline{\Omega})]^2$ of the subsequence is a strong solution to \eqref{app-vw}.
The estimate \eqref{vw0est} can be obtained from the estimate \eqref{lem28-est} with the weak $H^3$ convergence of $\{(V_{m_k},W_{m_k})\}_{k\in\mathbb{N}}$ and $\{(\partial_rV_{m_k},\partial_rW_{m_k})\}_{k\in\mathbb{N}}$.
Also, one can see that the solution is unique if ${\bm\tau}_p\in(0,{\bm\tau}^{\flat}]$ for a sufficiently small ${\bm\tau}^{\flat}\in(0,{\bm\tau_{\flat}}]$
depending only on the data and $(L,\delta_1,\delta_2,\delta_3)$
by using a standard contraction argument.
The proof of Lemma \ref{Lemma27} is completed.
\end{proof}

Back to the proof of Lemma \ref{lema-po}.
Let $(\chi_h,\Psi_{h})$ be the solution obtained in Lemma \ref{Lem-po}.
Then the boundary value problem \eqref{lin-pps-bd}-\eqref{fix-lin-pps}  has a unique solution $(\chi:=\chi_h+\chi_{bd},\Psi:=\Psi_h+\Psi_{bd})$ that satisfies 
\begin{equation}\label{fix-est-po}
\begin{split}
&\|\chi\|_{H_{\ast}^4(\Omega)}+\|\Psi\|_{H_{\ast}^4(\Omega)}\\
&\le C\left(\|S_{\ast}-S_0\|_{H_{\ast}^4(\Omega)}+\left\|\frac{\Lambda_{\ast}}{r\sin\phi}{\bf e}_{\theta}\right\|_{H_{\ast}^3(\Omega;\mathbb{R}^3)}+\|\psi_{\ast}{\bf e}_{\theta}\|_{H_{\ast}^5(\Omega;\mathbb{R}^3)}+{\bm\tau}_p\right)\\
&\quad+C\left(\|\chi_{\ast}\|_{H_{\ast}^4(\Omega)}^2+\|\Psi_{\ast}\|^2_{H_{\ast}^4(\Omega)}\right)\\
&\le C_{\sharp}\left((1+\delta_1+\delta_3){\bm\tau}_p+(\delta_2{\bm\tau}_p)^2\right)
\end{split}
\end{equation}
for a constant $C_{\sharp}>0$ depending only on the data and $L$.
For the unique solution $(\chi,\Psi)$ of \eqref{lin-pps-bd}-\eqref{fix-lin-pps} associated with $(\chi_{\ast},\Psi_{\ast})\in\mathcal{J}(\delta_2)$, we define an iteration mapping $\mathcal{I}_p:\mathcal{J}(\delta_2)\to[H_{\ast}^4(\Omega)]^2$ by 
\begin{equation*}
\mathcal{I}_p(\chi_{\ast},\Psi_{\ast})=(\chi,\Psi).
\end{equation*}
If we choose $\delta_2$ and ${\bm\tau}_2$ satisfying
\begin{equation*}
2C_{\sharp}(1+\delta_1+\delta_3)\le\delta_2\mbox{ and }
{\bm\tau}_2\le \frac{1}{2C_{\sharp}\delta_2},
\end{equation*}
then  $\mathcal{I}_p$ maps $\mathcal{J}(\delta_2)$ into itself.

Applying the Schauder fixed point theorem gives the mapping $\mathcal{I}_p$ has a fixed point $(\chi_{\sharp},\Psi_{\sharp})\in\mathcal{J}(\delta_2)$. 
Obviously, the fixed point $(\chi_{\sharp},\Psi_{\sharp})$ satisfies \eqref{pp-prob-sup} and \eqref{po=est}.
By applying the standard contraction argument,
one can prove that there exists a sufficiently small constant ${\bm\tau}_2\in(0,\frac{1}{2C_{\sharp}\delta_2}]$ depending only on the data and $(L, \delta_1,\delta_2,\delta_3)$ so that the solution is unique.
This finishes the proof of Lemma \ref{lema-po}.\qed


\bigskip
\noindent
{\bf Acknowledgements} \,\,
The research of Hyangdong Park was supported in part by the POSCO Science Fellowship of POSCO TJ Park Foundation and a KIAS Individual Grant (MG086702) at Korea Institute for Advanced Study. 

\bigskip
\noindent
{\bf Conflict of interests}\,\,
There is no conflict of interest.

\bigskip
\noindent
{\bf Data availability statement}\,\,
Data sharing not applicable to this article as no datasets were generated or analyzed during the current study.

\bibliographystyle{siam}
\bibliography{References}

\end{document}